\documentclass[12pt]{amsart}
\usepackage{amssymb}
\usepackage{amsfonts}
\usepackage{graphicx}
\usepackage{amsmath}
\usepackage{xcolor}

\newtheorem{theorem}{Theorem}

\newtheorem{conjecture}[theorem]{Conjecture}

\newtheorem{definition}[theorem]{Definition}
\newtheorem{example}[theorem]{Example}

\newtheorem{lemma}[theorem]{Lemma}

\newtheorem{proposition}[theorem]{Proposition}

\newcommand{\C}{\mathbb{C}}
\newcommand{\N}{\mathbb{N}}

\begin{document}

\title{A Fischer type decomposition theorem from the apolar inner product}
\author{ J. M. Aldaz* and H. Render}
\address{H. Render: School of Mathematical Sciences, University College
Dublin, Dublin 4, Ireland.}
\email{hermann.render@ucd.ie}
\address{J.M. Aldaz: Instituto de Ciencias Matem\'aticas (CSIC-UAM-UC3M-UCM)
and Departamento de Matem\'aticas, Universidad Aut\'onoma de Madrid,
Cantoblanco 28049, Madrid, Spain.}
\email{jesus.munarriz@uam.es}
\email{jesus.munarriz@icmat.es}
\maketitle

\begin{abstract} We continue the study initiated by H. S. Shapiro on Fischer decompositions of entire functions, 
showing that such decomposition exist in a weak sense (we do not prove uniqueness)
under hypotheses regarding the order of the entire function $f$ to be expressed as $f=  P\cdot q+r$, the polynomial $P$, and bounds on the apolar norm of homogeneous
polynomials of degree $m$. 

These bounds, previously used by Khavinson and Shapiro,
and by Ebenfelt and Shapiro, can be interpreted as a quantitative, asymptotic strengthening of Bombieri's inequality. In the special case where both the dimension of the space and the degree of $P$ are two, we characterize for which polynomials $P$ such bounds hold.

\end{abstract}

\section{ Introduction}

Let us denote by $\mathcal{P}\left(  \mathbb{C}^{d}\right)$ the set of all
polynomials with complex coefficients in the variable $z=\left(  z_{1},\dots,z_{d}\right)  \in
\mathbb{C}^{d}$,  and by $\mathcal{P}_{m}\left(
\mathbb{C}^{d}\right)$ the subspace of all homogeneous polynomials of degree
$m.$ Recall that a polynomial $P\left(  z \right)  $ is \emph{homogeneous} of
degree $\alpha$ if $P\left(  t z \right)  =t^{\alpha}P\left(  z \right)  $ for
all $t>0$ and  all $z \in \mathbb{C}^{d}.$ In particular, the zero polynomial is homogeneous of every degree. 
Given a polynomial $P\left(  z\right)  $, we denote by
$P^{\ast}\left(  z\right)  $ the polynomial obtained from $P\left(  z\right)
$ by conjugating its coefficients, and by $P\left(  D\right)  $  the linear
differential operator obtained by replacing the variable $z_{j}$ with the
differential operator $\frac{\partial}{\partial z_{j}}$.

It is well known that a polynomial $P\left( z\right) $ of degree $k$
can be written as a sum of homogeneous polynomials $P_{j}\left( z\right) $
of degree $j$ for $j=0, \dots ,k,$ so 
\begin{equation*}
P\left( z \right) =P_{k}\left( z \right) +\cdots +P_{0}\left( z \right) ,
\end{equation*}
and we call the homogeneous polynomial $P_{k}\left( z \right) $ the {\em leading
term}. 
Often the alternative notation $P\left( z \right) =P_{k}\left( z \right) - \cdots - P_{0}\left( z \right) $ will be more convenient, in relation with formulas regarding division of polynomials.

\vskip .2 cm

Fischer's theorem states that if $P$ is a  polynomial with leading term $P_k$, then 
the following decomposition holds: for every polynomial $f\in\mathcal{P}\left(
\mathbb{C}^{d}\right)  $ there exist \emph{unique} polynomials $q\in
\mathcal{P}\left(  \mathbb{C}^{d}\right)  $ and $r\in\mathcal{P}\left(
\mathbb{C}^{d}\right)  $ such that
\begin{equation} \label{dual}
f=  P\cdot q+r\text{ and }P_k^{\ast}\left(  D\right)  r=0.
\end{equation}

 To indicate the dependency of $q$ on $f$ and $P$ we will often write $q = T_P(f)$. H. S. Shapiro  studied Fischer
decompositions in a wider setting, cf. \cite{Shap89}, going beyond the case of polynomials to
more general function spaces, in particular, to the space $E\left( \mathbb{C}
^{d}\right) $ of all entire functions $f:\mathbb{C}^{d}\rightarrow \mathbb{C}
$. It is useful to adopt a notion introduced in \cite[p. 522]{Shap89}:
suppose that $E$ is a vector space of infinitely differentiable functions $
f:G\rightarrow \mathbb{C}$ (defined on an open subset $G$ 
of $\mathbb{C}^{d}$) that is a module over $\mathcal{P}\left( \mathbb{C}
^{d}\right) $.

\begin{definition} \label{FP} We say that a polynomial $P$ and a differential operator 
$Q\left( D\right) $ form a \emph{Fischer pair $(P, Q)$ for the space }$E$, if for
each $f\in E$ there exist \emph{unique} elements $q\in E$ and $r\in E$ such
that 
\begin{equation}
	f=P\cdot q+r\text{ and }Q\left( D\right) r=0.  \label{eqDecomp}
\end{equation}
We will speak of {\em weak Fischer pairs} when the decomposition $f = P\cdot q+r$ is not assumed to be unique. But the expression ``Fischer decomposition'' will be used even in the absence of uniqueness.
\end{definition}
Shapiro proved in \cite[Theorem 1]{Shap89} that the following version of Fischer's theorem is 
true when $\mathcal{P}\left( \mathbb{C}^{d}\right) $ is replaced by $E\left( 
\mathbb{C}^{d}\right) $: for every {\em homogeneous} polynomial $P$ and every entire function $f$ there exist \emph{unique} entire functions $q$ and $r$ such that
\[
f=  P \cdot q+r\text{ and }P^{\ast}\left(  D\right)  r=0.
\]
That is, $P$ and $P^*$ form a Fischer pair for the entire functions, whenever $P$ is homogeneous. Shapiro conjectured that his theorem held even for non-homogeneous $P$ and arbitrary entire functions, cf. \cite[p. 517 (i)]{Shap89}.  Below we obtain some partial results for an  arbitrary polynomial $P$ with leading term $P_k$, showing that $(P, P^*_k)$ is  a weak Fischer pair  for  entire functions of sufficiently low order.

The following related conjecture, for the special case where the polynomial under consideration is $ P_k - 1$, has been studied in the literature:

\vskip .2 cm

{\bf Conjecture:}
{\em Let $P$ be a nonconstant homogeneous polynomial and define
\[
F_P \left(  q\right)  := P^{\ast}\left(  D\right)  \left[  \left(  P-1\right)
q\right]  .
\]
Then $F_{P}$ is a bijection on the set of all entire functions; equivalently, $P-1$
and $P^{\ast}\left(  D\right)  $ form a Fischer pair for $E\left( 
\mathbb{C}^{d}\right) $.}

\vskip .2 cm

We mention that the equivalence between the 
bijectivity of  the Fischer operator $F_{P}$ and the fact that 
 $P-1$
and $P^{\ast}\left(  D\right)  $ form a Fischer pair, is due to Meril and Struppa, cf. \cite[Proposition 1]{MeSt85}.

\vskip .2 cm

In 	\cite{KhSh92}, D. Khavinson and H.S. Shapiro
prove a partial case of the preceding conjecture,  restricting the class of entire functions under consideration and assuming uniqueness from the outset, cf. \cite[Theorem 3]{KhSh92} (we mention that the proof uses Fredholm theory). 

The homogeneous polynomials $P$ in \cite[Theorem 3]{KhSh92} satisfy a certain property called ``amenability", defined as follows in \cite[p. 464]{KhSh92}:   a homogeneous
	polynomial $P_{k} \in \C[z_1, \dots , z_d]$ of degree $k > 0$ is said to be {\em amenable} if for each 
$j\in  \{1,\dots, d\}$ there is a multi-index $\alpha$, with $|\alpha | = k -1$, such that $D^\alpha P_k$ is a non-zero constant times $z_j$ (note however that no polynomial of degree  $k = 1$ is amenable when $d \ge 2$).
 Khavinson and Shapiro observe that amenability holds in the important special case
$P=\left\vert x\right\vert ^{2}$, so
$P^{\ast}\left(  D\right)  =\Delta$. Furthermore, amenability entails the following bounds, cf. \cite[Lemma 11]{KhSh92}, which represent a weaker assumption:
if $P$  is an amenable homogeneous polynomial and $f$ is homogeneous of degree $m$, that is, $f \in \mathcal{P}_{m}\left(
\mathbb{C}^{d}\right)$, then there is a  constant $C > 0$, independent of $m$, such that
\begin{equation}
\label{KhaSha}
\|P f\|_a \ge C m^{1/2} \|P\|_a \|f\|_a, 
\end{equation}
where $\|\cdot \|_a$ is the norm defined by the apolar inner product (see the next section).
Note that the term $ \|P\|_a$ can be absorbed into the constant $C$.

Satisfaction of the Khavinson-Shapiro bounds (\ref{KhaSha}) appears to be a much more generic
property than amenability. We explore this question in the specific instance where
both  the dimension and the degree are  2, so $P\left(  z_1 ,z_2 \right)  =a z_1^{2}+ b  z_1 z_2 + c z_2^{2}$ (with $a,b,c$ not all 0). While in this case the only amenable polynomials are $ a  z_1^{2} + c z_2^{2}$ when $a, c \ne 0$, and $ b z_1 z_2$ when $b \ne 0$, we prove that the Khavinson-Shapiro bounds are satisfied precisely when 
 $4 a c \ne b^2$, cf. Theorem \ref{twotwo} below.

Recalling  Bombieri's inequality
$$\|P f\|_a \ge \|P\|_a \|f\|_a,
$$
we see that satisfaction of the Khavinson-Shapiro bounds entails a quantitative strengthening of Bombieri's bounds, as $m \to \infty$.

 Let $P_k$ be a nonzero homogeneous polynomial of degree $k$.  By Fischer's Theorem, for each
homogeneous polynomial $f_{m}$ of degree $m$ there exist unique
polynomials  $T_{P_k}$ and $r_{m}$, with $P_k^{\ast}\left(  D\right)  r_{m}=0$ and
\begin{equation} \label{linealT}
f_{m}=P_k\cdot T_{P_k}\left(  f_{m}\right)  + r_{m}.
\end{equation}
This decomposition is orthogonal with respect to the apolar inner product, since by (\ref{eqFischeradjoint}) below
\[
\langle  P_k \cdot T_{P_k}\left(  f_{m}\right)  ,r_{m}\rangle_a  =\langle  T_{P_k}\left(
f_{m}\right)  ,P_k^{\ast}\left(  D\right)  r_{m}\rangle_a  =0.
\]
By the Pythagorean Theorem, and under the assumption that the  Khavinson-Shapiro bounds (\ref{KhaSha}) hold, we obtain
\[
\left\Vert f_{m}\right\Vert _{a}^{2}=\left\Vert P_k\cdot T_{P_k}\left(
f_{m}\right)  \right\Vert _{a}^{2}+\left\Vert r_{m}\right\Vert ^{2}
\geq\left\Vert P_k \cdot T_{P_k }\left(  f_{m}\right)  \right\Vert _{a}^{2} \ge C^2  m  \| T_{P_k }\left(  f_{m}\right) \|_a^2.
\]
This shall be the type of assumption used in the theorem stated next. The main differences with  the Khavinson-Shapiro result lie in the fact  that arbitrary polynomials $P = P_{k}-P_{k - 1}- \cdots - P_{0}$ are used for division, instead of polynomials of the form
$P = P_{k}-1$, and instead of bounds of the form $O(m^{1/2})$ we use $O(m^{\tau/2})$, where $\tau \ge 0$.

\vskip .2 cm

Next we state the main result of the paper.

\begin{theorem}\label{main}
Let $P_{k}$ be a homogeneous polynomial of degree $k > 0$ on $\C^d$, and let us write
$T:=T_{P_{k}}$, where $T_{P_{k}}$ is defined by (\ref{linealT}). Assume that there exist a $C=C(P_{k})>0$ ($C$ independent of
$m$) and an $\tau\in \{0, \dots , k\}$, 
  such that for every $m>0$ and 
every homogeneous polynomial $f_{m}$ of degree $m,$ the following inequality
holds:
\begin{equation}
\left\Vert Tf_{m}\right\Vert _{a}\leq\frac{C}{m^{\tau/2}}\left\Vert
f_{m}\right\Vert _{a} \label{eqT}.
\end{equation}
If for $0\leq j<k$ the polynomials $P_{j}\left(   z\right)  $ are
homogeneous of degree $j$,  and for some $\beta<k$ and
every $j$ with $\beta <j<k$ we have $P_{j}=0$, then
for every entire function $f: \C^d\to \C$ of order 
$\rho$, where $\rho$ satisfies the inequality
$$
\rho (k-\tau) < 2(k-\beta),
$$
there exist entire functions $q$ and $r$ of order $\leq\rho$ with
\[
f=\left(  P_{k}-P_{\beta}- \cdots  -P_{0}\right)  q+r\text{ and }P_{k}^{\ast
}\left(  D\right)  r=0.
\]
\end{theorem}

Throughout this paper, given the homogeneous polynomial $P_k$ we will use the abbreviation $T:=T_{P_{k}}$.

Note that no statement is made regarding the possible uniqueness of the entire functions $q$ and $r$. However, we shall see that when $k=1$, the conjecture is true and furthermore we do have uniqueness, even though the Khavinson-Shapiro bounds do not hold, cf. Theorem \ref{linealconj} below.
When the dimension is $d = 1$, it follows directly from polynomial interpolation that
the result also holds, cf. Theorem \ref{dim1}: 
for every entire function $f:\C\to \C$,
there exist unique entire functions $q$ and $r$ such that
\[
f=\left(  P_{k}-P_{\beta}- \cdots  -P_{0}\right)  q+r\text{ and }P_{k}^{\ast
}\left(  D\right)  r=0.
\]
Thus, in the proof of Theorem \ref{main} above, it can always be assumed that $k, d \ge 2$.

Finally let us mention that 
conditions of the type  (\ref{KhaSha}) 
 had already been used in the work of P. Ebenfelt and
H.S. Shapiro about a generalized Cauchy-Kowaleskaya theorem in \cite{EbSh95}, and the
mixed Cauchy problem for differential equations in \cite{EbSh96}, see also \cite{EbRe08}, \cite{EbRe08b}. Let  $P$ be a polynomial of degree $k,$ written as a sum of homogeneous
polynomials
\[
P=P_{k}+P_{k-1}+\cdots+P_{0}%
\]
with leading coefficient $P_{k}$, and let the floor function $\lfloor  x\rfloor$ be the largest integer $\leq x$. Theorem 3.1.1 in [6] (see also remark on p. 259 in [6]) implies that for each entire function $f$, there
exist entire functions $q$ and $r$ such that
\begin{equation}
	f=P_{k}q+r\text{ and }P^{\ast}\left(  D\right)  r=0\label{eqdual},
\end{equation}
provided, first, that there are constants $C>0$ and $\tau\geq0$ such
that
\begin{equation}
	\left\Vert P_{k}f_{m}\right\Vert \geq Cm^{\tau/2}\left\Vert f_{m}\right\Vert
	_{a}\label{eqtau1},
\end{equation}
and second,  that $P_{j}=0$  for all $j$ with
\[
\left\lfloor  \frac{k+\tau}{2}\right\rfloor  <j<k.
\]
 Note that
(\ref{eqdual}) can be seen as the dual problem to the one formulated in (\ref{dual}).

\section{Basic properties of the apolar inner product}

The variant of the Khavinson-Shapiro result obtained here also employs   the apolar inner product (introduced during the XIX century within the theory of invariants) also known as  Fischer's inner product, and, with different normalizations, as Bombieri's inner product (more details can be found in \cite{Rend08}).  For easy reference we include some known information about it.

Denote the natural numbers by $\N_0$ (emphasizing the fact that they include zero) and recall
 the standard notation for multi-indices: given $\alpha=\left(  \alpha
_{1},\dots ,\alpha_{d}\right)  \in\mathbb{N}_{0}^{d}
$, we write   $z^{\alpha
}=z_{1}^{\alpha_{1}}\cdots  z_{d}^{\alpha_{d}},$  $\alpha!=\alpha_{1}
! \cdots \alpha_{d}!$, and $\left\vert \alpha\right\vert =\alpha_{1}+\cdots
+\alpha_{d}$.

\vskip .2 cm
Let $P$ and $Q$ be polynomials of degree $N$ and $M$, respectively given by
\[
P\left(  z \right)  =\sum_{\alpha\in\mathbb{N}_{0}^{d},\left\vert
\alpha\right\vert \leq N}c_{\alpha}z ^{\alpha}\text{ and }Q\left(  z\right)
=\sum_{\alpha\in\mathbb{N}_{0}^{d},\left\vert \alpha\right\vert \leq
M}d_{\alpha}z^{\alpha}.
\]
The \emph{apolar inner product} $\langle  \cdot,\cdot\rangle_{a}$ on
$\mathcal{P}\left(  \mathbb{C}^{d}\right)  $ is defined by
\begin{equation}
\left\langle P,Q\right\rangle_{a} :=\left[  Q^*\left(
D\right)  P\right]  \;(0)=\sum_{\alpha\in\mathbb{N}_{0}^{d}}\alpha!c_{\alpha
}\overline{d_{\alpha}}\label{eqQPD},
\end{equation}
and the associated {\em apolar norm}, by
$
\left\Vert f\right\Vert _{a}=\sqrt{\left\langle f,f\right\rangle_{a}}.
$
Note that for the constant polynomial $1$ one has
\begin{equation}
\left\langle P,1\right\rangle_{a}=P\left(  0\right)  .\label{eqFischer2}
\end{equation}
A fundamental property of the apolar inner product, is that the adjoint operator of
the multiplication operator $M_{Q}\left(  g\right)  =Qg$ is the differential
operator associated to the polynomial $Q^{\ast}.$

For the reader's convenience we include the short proof of the following known result.

\begin{proposition}
The following formulae hold for polynomials $P,Q$ and $f,g$:
\begin{equation}
\left\langle P,Q\right\rangle_{a}
=
[Q^{\ast}\left(  D\right)  P ] \left(  0\right)
=
\left\langle Q^{\ast}\left(  D\right)
P,1\right\rangle_{a}.
\label{eqFischerid}
\end{equation}
and
\begin{equation}
\left\langle Q^{\ast}\left(  D\right)  f,g\right\rangle_{a}=\left\langle
f,Q\cdot g\right\rangle_{a} 
\label{eqFischeradjoint}.
\end{equation}

\end{proposition}

\begin{proof}
Regarding  (\ref{eqFischerid}), the first equality is the definition and the second follows from (\ref{eqFischer2}). From (\ref{eqFischerid}) we conclude that
\begin{align*}
\left\langle Q^{\ast}\left(  D\right)  f,g\right\rangle_{a}  &  =\left\langle
g^{\ast}\left(  D\right)  Q^{\ast}\left(  D\right)  f,1\right\rangle_{a},  \mbox{ \  and }
\\
\left\langle f,Q\cdot g\right\rangle_{a}  &  =\left\langle Q^{\ast}\left(
D\right)  g^{\ast}\left(  D\right)  f,1\right\rangle_{a}.
\end{align*}
Since $g^{\ast}\left(  D\right)  Q^{\ast}\left(  D\right)  =Q^{\ast}\left(
D\right)  g^{\ast}\left(  D\right)  $  equation
(\ref{eqFischeradjoint}) follows.
\end{proof}
\vskip .2 cm
\vskip .2 cm

It is shown in \cite[Theorem 38]{Rend08} that if $(P_k,Q)$ is a Fischer pair, where $P_k$ is the principal part
of $P$ and $Q$ is homogeneous, then $(P,Q)$ is a Fischer pair. However, from the fact that  $(P,Q)$ is a Fischer pair we cannot conclude that $(P_k,Q)$ is a Fischer pair, as the following example from \cite{MeSt85} shows (cf.  also
\cite[Example 34]{Rend08}): take $n = 2$, $P(z_1, z_2) = z_1 - z_2^2$ and $Q(z_1, z_2) = z_1$. Then it can be checked that $P$ and $Q$ form a Fischer pair, while $P_k(z_1,z_2) = - z^2$ and $Q$
do not, since they are homogeneous of different degrees. This is impossible by \cite[Theorem 36]{Rend08}.

\section{Apolar norms of products of linear factors}
In addition to $\partial /\partial z_j$, we also use $D_j$ and $\partial_j$ to denote the $j$-th partial derivative.
The results in this section will be used later, when comparing amenability to the Khavinson-Shapiro bounds.
For $a,b\in\mathbb{C}^{d}$ we denote by $\left\langle a,b\right\rangle
$ the standard inner product
\[
\left\langle a,b\right\rangle =a_{1}\overline{b_{1}}+\cdots+a_{d}
\overline{b_{d}} = \overline{\langle b, \overline{a}\rangle}.
\]
Given $b, z \in\mathbb{C}^{d}$, consider the lineal
polynomial
\[
Q_b\left(  z\right)  :=\left\langle z, \overline{b}\right\rangle =b_{1}z_{1}+\cdots
+b_{d}z_{d},
\]
and write
\[
L_{\overline{b}}  = Q_b^{\ast}\left(  D\right)  =\overline{b_{1}}\partial_1+\cdots+\overline{b_{d}}\partial_{d}.
\]
It is easy to see that for any pair of differentiable functions $f\left(
z\right)  ,g\left(  z \right)  $,
\[
L_{\overline{b}}\left(  fg\right)  =
{\displaystyle\sum_{j=1}^{n}}
\overline{b_{j}}\partial_{j}\left(  f\cdot g\right)
= (L_{\overline{b}}f )\cdot g+f\cdot L_{\overline{b}}g.
\]
In the next result we assume that the nonzero vector $c$ is orthogonal to the vectors
$a_{1},\dots ,a_{M}$  (they may be linearly dependent, or even repeating):

\begin{proposition} \label{lineal}
Assume that $a_{1},\dots ,a_{M}\in\mathbb{C}^{d}$, where $a_{k}=\left(  a_{k,1},\dots ,a_{k,d}\right)  $, and set $\sigma_{k}\left(
 z \right)  =\left\langle z, \overline{a_{k}} \right\rangle $ for $k=1,\dots ,M.$ Suppose 
there exists a vector $c\in\mathbb{C}^{d}\setminus \{0\}$ such that for $k=1,\dots ,M$,
\[
{\displaystyle\sum_{j=1}^{d}}
\overline{a_{k,j}}c_{j}=0.
\]
 Let $g$ be a univariate polynomial. Then
\[
\left\Vert \sigma_{1} (z) \cdots\sigma_{M} (z)  \cdot g\left(  \left\langle z,
\overline{c} \right\rangle \right)  \right\Vert _{a}^{2}=\left\Vert \sigma_{1} (z) 
\cdots\sigma_{M} (z) \right\Vert _{a}^{2}\left\Vert g\left(  \left\langle z,
\overline{c} \right\rangle \right)  \right\Vert _{a}^{2}.
\]
Taking $g\left(  t\right)  =t^{m}$ we see that
\[
\left\Vert \sigma_{1} (z) \cdots\sigma_{M} (z) \cdot\left\langle z, \overline{c} \right\rangle
^{m}\right\Vert _{a}^{2}=\left\Vert \sigma_{1} (z) \cdots\sigma_{M} (z) \right\Vert
_{a}^{2}\left\Vert \left\langle z, \overline{c} \right\rangle ^{m}\right\Vert _{a}^{2}.
\]
so the product $\sigma_{1} (z) \cdots\sigma_{M} (z) $ does not satisfy the Khavinson-Shapiro bounds.
\end{proposition}

\begin{proof} If $a =\left(  a_{1},\dots ,a_{d}\right)  $ and $c$ are orthogonal, so
 $\langle a, c \rangle = 0$, then 
 \[
L_{\overline{a}}\left(  g\left(  \left\langle z, \overline{c} \right\rangle \right)
\right)  =
{\displaystyle\sum_{j=1}^{d}}
\overline{a_{j}} \partial_{j}\left(  g\left(  \left\langle z, \overline{c} \right\rangle \right)  \right)  =g^{\prime}\left( \left\langle z, \overline{c} \right\rangle\right)
{\displaystyle\sum_{j=1}^{d}}
\overline{a_{j}}c_{j}=0.
\]
 It
follows that 
\[
L_{\overline{a}}\left(  f\cdot g\left(  \left\langle z, \overline{c} \right\rangle \right)
\right)  =L_{\overline{a}}f\cdot g\left(  \left\langle z, \overline{c} \right\rangle
\right)  +f\cdot L_{\overline{a}}\left(  g\left(  \left\langle z, \overline{c} \right\rangle \right)  \right)  =L_{\overline{a}}f\cdot g\left(
\left\langle z, \overline{c} \right\rangle \right).
\]
Since $\sigma_{k}\left(
 c \right)  =\left\langle c , \overline{a_{k}}\right\rangle = 0$ for $k=1,\dots ,M$, we have
\[
L_{\overline{a_{1}}}\cdots L_{\overline{a_{M}}}\left(   \sigma_{1} (z) \cdots\sigma_{M} (z)  \cdot g\left(  \left\langle z, \overline{c} \right\rangle \right)  \right)
=L_{\overline{a_{1}}}\cdots L_{\overline{a_{M}}}\left(   \sigma_{1} (z) \cdots\sigma_{M} (z) \right)  \cdot g\left(  \left\langle z, \overline{c} \right\rangle \right).
\]
Furthermore, we know that
\begin{align*}
\left\Vert  \sigma_{1} (z) \cdots\sigma_{M} (z)  \cdot g\left(  \left\langle z, \overline{c} \right\rangle \right)  \right\Vert _{a}^{2}  &  =\left\langle
L_{\overline{a_{1}}}\cdots L_{\overline{a_{M}}}\left(  \sigma_{1} (z) \cdots\sigma_{M} (z) \cdot g\left(  \left\langle z, \overline{c} \right\rangle \right)  \right)
,g\left( \left\langle z, \overline{c} \right\rangle \right)  \right\rangle _{a}\\
&  =L_{\overline{a_{1}}}\cdots L_{\overline{a_{M}}}\left(   \sigma_{1} (z) \cdots\sigma_{M} (z) \right)  \left\Vert g\left( \left\langle z, \overline{c} \right\rangle
\right)  \right\Vert _{a}^{2}\\
&  =\left\Vert  \sigma_{1} (z) \cdots\sigma_{M} (z) \right\Vert _{a}^{2}\left\Vert
g\left( \left\langle z, \overline{c} \right\rangle \right)  \right\Vert _{a}^{2}.
\end{align*}
\end{proof}

\section{Observations on amenability and the  bounds of Khavinson and Shapiro}

Let $P\in \C[z_1, \dots , z_d]$. We say that $P$ is independent of $z_j$ if $\partial /\partial z_j P = 0$.  If for
no $j$ this holds, then we say that $P$ depends on all the variables. Obviously if $P_k$ is amenable then it depends on all the variables. The converse is not true when $d >1$: for $k= d = 2$, just consider $P_2 (z_1, z_2) = z_1^2 + z_1 z_2 + z_2^2$;
the only possibilities for $\alpha$ are $(1,0)$ and $(0,1)$; clearly, neither of them satisfy the amenability condition. This example also shows that a symmetric polynomial (i.e., invariant under permutations of the variables) may fail to be amenable.

\vskip .2 cm

Next we  explore some cases of homogeneous polynomials that satisfy, or fail to satisfy, the Khavinson-Shapiro bounds (\ref{KhaSha}).  Formula (\ref{reznick}) below, where the sum is taken over all multi-indices with non-negative entries, appears in \cite[p. 523]{Shap89} (an earlier, and more general statement applying to certain entire functions that include the polynomials, can be found in \cite[Theorem 3]{NeSh66}; in \cite{Zeil} formula (\ref{reznick}) is called Reznick identity). Note that when $|\alpha | > k$ the corresponding terms are zero, while summing only over the multi-indices with $|\alpha | = k$ yields Bombieri's inequality. Refinements of Bombieri's inequality can be obtained by estimating  other terms in the sum
\begin{equation} \label{reznick}
\|P_k f_m \|_a^2 = \sum_\alpha  \| \partial^\alpha P_k^*(D) f_m \|_a^2/\alpha!,
\end{equation}
as we do next in some special cases. 
Suppose the analogous statement to the Khavinson-Shapiro bounds holds for some positive integer $\tau$ not necessarily 
equal to one, i.e., for every homogeneous polynomial  $f_m$  of degree $m$,  there is a  constant $C > 0$, independent of $m$, such that
\begin{equation}
\label{alphaKhaSha}
\|P f\|_a \ge C m^{\tau/2}  \|f_m\|_a.
\end{equation}
Then we call $\tau$ a {\em Khavinson-Shapiro exponent}, and the largest such $\tau$ the {\em optimal Khavinson-Shapiro exponent}.

\begin{theorem} \label{KhaShabounds}
	Let $P_{k}: \C^d \to\C$ be a homogeneous
	polynomial of degree $k > 0$. Then:
	
	\vskip .2 cm

	1) When the dimension $d = 1$,  then the optimal Khavinson-Shapiro exponent is always equal to $k$.
	
	2) When the degree $k = 1$ and the dimension $d > 1$, no homogeneous polynomial $P_1$ satisfies the Khavinson-Shapiro bounds. 
	For $k > 1$, the optimal  Khavinson-Shapiro exponent is bounded by $k-1$. 
	
	3) When all the exponents of the variables are even (and in particular, so is the total degree $k$), the homogeneous polynomial $P_k$ is amenable, and hence it satisfies the Khavinson-Shapiro bounds. 
			
\end{theorem}

\begin{proof} When $d = 1$, writing $P_k(z) = a z^k$ and  $f_m(x) = b z^m$ with $a, b \ne 0$, 
we have 
\begin{equation} \label{1d}
\|P_k f_m \|_a^2 = | a |^2 \frac{(k + m) !}{m!}   \|f_m\|_a^2. 
\end{equation}
For 2), suppose that 
$$
P_1\left(  z\right)  = \sum_{1 \le i \le d}c_{i} z_{i},
$$
 where some but not all the coefficients $c_i$ might be $0$, and let $m > 1$. Since $P_1(D)$ maps the homogeneous polynomials of degree $m$ into those of degree $m - 1$, and the latter space has a smaller dimension (because $d > 1$)
 it follows that there is a homogeneous polynomial $f_m \ne 0$ of degree $m$ such that $P_1(D) f_m = 0$.
 By the Newman-Shapiro identity (\ref{reznick}), a.k.a. Reznick identity, with $|\alpha | = 1$, 
 $$
\|P_1 f_m \|_a^2 
= 
 \|P_1 \|_a^2 \|f_m\|_a^2 +   \|P_1(D) f_m \|_a^2 =  \|f_m\|_a^2 \|P_1 \|_a^2. 
$$ 
For $k > 1$,  the same argument shows that there is a homogeneous polynomial 
$f_m \ne 0$ of degree $m$ such that $P_k(D) f_m = 0$, so the term in (\ref{reznick}) corresponding to the multi-index $\alpha = 0$ vanishes, and
hence the optimal Khavinson-Shapiro exponent is bounded by $k-1$.

Next we check the assertion in 3). For each 
$j\in  \{1,\dots, d\}$ choose a monomial $c_\alpha z^\alpha$ with maximal $\alpha_j$, and let $\alpha^\prime$ be obtained from $\alpha$ by replacing $\alpha_j$ with $\alpha_j - 1$, so $|\alpha^\prime| = k -1$. Then for some constat $c \ne 0$ we have $D^{\alpha^\prime} c_\alpha z^\alpha = c z_j$.
If $c_\beta z^\beta$ is another  monomial with $\beta_j < \alpha_j$, so  $\beta_j \le \alpha_j -2$, then $D^{\alpha^\prime} c_\beta z^\beta = 0$, while if
 $\beta_j = \alpha_j$ then for some $i \ne j$, $\beta_i < \alpha_i$, so again $D^{\alpha^\prime} c_\beta z^\beta = 0$. It follows that $D^{\alpha^\prime} P_k$ is a non-zero constant times $z_j$.
\end{proof}

\vskip .2 cm

Note that when the remainder $r = 0$, we have $f_m = P_k Tf_m$,  so by Beauzamy's inequality (Lemma \ref{lemmaD} below) we always have $\tau \le k$ in (\ref{eqT}).

\begin{example} By Lemma \ref{lemmaD} below, for every $f_m$ we  have 
$$
\|P_k f_m \|_a^2 
\le 
C m^{k} \|f_m\|_a^2,
$$ 
so $\tau \le k$, and by part 2) of the preceding theorem, if $d > 1$ it is always possible to find an $f_m$ for which  $\tau \le k - 1$. The condition 
\begin{equation}\label{kTf}
\left\Vert Tf_{m}\right\Vert _{a}\leq\frac{C}{m^{\frac{\tau}{2}}}\left\Vert
f_{m}\right\Vert _{a}
\end{equation}
used in Theorem \ref{main} can be much weaker when applied to specific functions. In the extreme case $ Tf_{m} = 0$ obviously every
positive $\tau$ will work. For a less extreme example, take $m, \ell \gg k$,  and let $P_{k} \in \C[z_1 , z_d]$ be $z_1^k$.
Choosing $f_m (z_1, z_2) := z_1^k z_2^{m-k} + m^\ell z_2^{m}$, it is clear that we can take $\tau = \ell$. 
But as noted above, when the remainder $r = 0$ we have $f_m = P_k Tf_m$, so since $\tau$ is required to work for all homogeneous polynomials,   $\tau \le k$ also in (\ref{kTf}), and when $d > 1$, $\tau \le k - 1$.
\end{example}

\vskip .2 cm

Next we show that if $k = d = 2$, so $P\left(  z_1 ,z_2 \right)  =a z_1^{2}+ b  z_1 z_2 + c z_2^{2}$, 
then  $P$ will satisfy the Khavinson-Shapiro bounds under the  condition $4 a b \ne b^2$, which is distinctly more general than amenability. It is easy to check
 that for $k = d = 2$, the only amenable polynomials are $Q\left(  z_1 ,z_2 \right)  = a  z_1^{2} + c z_2^{2}$ when $a, c \ne 0$ and $R\left(  z_1 ,z_2 \right)  = b z_1 z_2$ when $b \ne 0$.

\begin{theorem} \label{twotwo}
	Let $P\left(  z_1 ,z_2 \right)  =a z_1^{2}+ b  z_1 z_2 + c z_2^{2}$ be a 
	polynomial with complex coefficients $a,b,c$, not all of them 0. Then the following are equivalent:
	
	1) $P$ does not satisfy  the Khavinson-Shapiro bounds.
	
	2) $4 a c = b^2$. 
	
	3) 	$P\left(  z_1, z_2\right) =  \left(  r z_1 + s z_2\right)
	^{2}$, where the only condition  on the complex coefficients $r$ and  $s$ is that they
	cannot simultaneously be 0.

\end{theorem}

\begin{proof} To obtain 1) $\implies$ 2), we shall prove an equivalent formulation: if $4 a b \ne b^2$, then $P$ satisfies the Khavinson-Shapiro bounds.
	
	\vskip .2cm

Given $P\left(  z_1 ,z_2 \right)  =a z_1^{2}+ b  z_1 z_2 + c z_2^{2}$,  we
have
\[
\partial_1 P^{\ast}\left(   z_1 ,z_2  \right)  =2\overline{a
}z_1+\overline{b}z_2\text{ and }\partial_2 P^{\ast}\left(
 z_1 ,z_2  \right)  =\overline{b}z_1+2\overline{c}z_2.
\]

Let 
$$
f_m (z_1,z_2) 
=
\sum_{0 \le i \le m}   c_{i} z_1^i z_2^{m - i}.
$$
A computation (alternatively, use \cite[Lemma 3]{Shap89}) shows that 
$$
m \|f_m\|_a^2 = \|\partial_1 f_m\|_a^2 +\|\partial_2 f_m\|_a^2.
$$
Let $0 \le t \le 1$ satisfy
\begin{equation} \label{proportion}
t m \|f_m\|_a^2 = \|\partial_1 f_m\|_a^2, \mbox{ \ so \ } (1 - t) m \|f_m\|_a^2 = \|\partial_2 f_m\|_a^2.
\end{equation}
It follows from (\ref{reznick}) that
\[
\left\Vert Pf_{m}\right\Vert _{a}^{2}\geq A_{m},\text{ where }A_{m}:=
{\displaystyle\sum_{k=1}^{2}}
\left\Vert \left(  \partial_k P^{\ast}\right)  \left(
D\right)  f_{m}\right\Vert _{a}^{2}.
\]
Now
\begin{align*}
\left\Vert \left(  2\overline{a}\partial_1 
+\overline{b}\partial_2\right)  f_{m}\right\Vert
_{a}^{2}  &  =4\left\vert a\right\vert ^{2}\left\Vert \partial_1 f_{m}\right\Vert _{a}^{2}+\left\vert b\right\vert
^{2}\left\Vert \partial_2  f_{m}\right\Vert _{a}^{2}\\
&  +2\overline{a}b\left\langle \partial_1 f_{m}
,\partial_2 f_{m}\right\rangle _{a}+2a\overline{b
}\left\langle \partial_2 f_{m},\partial_1
 f_{m}\right\rangle _{a},
\end{align*}
and the sum of the last two terms is
\[
4\operatorname{Re}\left(  \overline{a}b\left\langle \partial_1 f_{m},\partial_2 f_{m}\right\rangle _{a}\right).
\]
The analogous formula holds for $\left\Vert \left(  \overline{b}\partial_1  +2\overline{c}\partial_2 \right)
f_{m}\right\Vert _{a}^{2}.$ It follows that
\begin{align*}
A_{m}  &  =\left\Vert \left(  2\overline{a}\partial_1 +\overline{b}\partial_2 \right)  f_{m}\right\Vert
_{a}^{2}+\left\Vert \left(  \overline{b}\partial_1 +2\overline{c}\partial_2 \right)  f_{m}\right\Vert
_{a}^{2}\\
&  =\left(  4\left\vert a\right\vert ^{2}+\left\vert b\right\vert
^{2}\right)  \left\Vert \partial_1  f_{m}\right\Vert _{a}
^{2}+\left(  \left\vert b\right\vert ^{2}+4\left\vert c\right\vert
^{2}\right)  \left\Vert \partial_2 f_{m}\right\Vert _{a}
^{2}+\\
&  +4\operatorname{Re}\left(  \left(  \overline{a}b+\overline{b
}c\right)  \left\langle \partial_1 f_{m},  \partial_2 f_{m}\right\rangle _{a}\right)
\end{align*}
Using  $\operatorname{Re}z\geq-\left\vert z\right\vert $ and the
Cauchy-Schwarz inequality, we conclude that
\[
A_{m}\geq\left(  4\left\vert a\right\vert ^{2}+\left\vert b
\right\vert ^{2}\right)  \left\Vert \partial_1 f_{m}
\right\Vert _{a}^{2}
\]
\[+\left(  \left\vert b\right\vert ^{2}+4\left\vert
c\right\vert ^{2}\right)  \left\Vert \partial_2
f_{m}\right\Vert _{a}^{2}-4\left\vert \overline{a}b+\overline{b
}c\right\vert \left\Vert \partial_1 f_{m}\right\Vert
_{a}\left\Vert \partial_2 f_{m}\right\Vert_a.
\]
It now follows from (\ref{proportion}) that
\[
A_{m}\geq\left(  4\left\vert a\right\vert ^{2}
+\left\vert b
\right\vert ^{2}\right)  mt\left\Vert f_{m}\right\Vert _{a}^{2}
\]
\[
+\left(
\left\vert b\right\vert ^{2}+4\left\vert c\right\vert ^{2}\right)
m\left(  1-t\right)  \left\Vert f_{m}\right\Vert _{a}^{2}-4\left\vert
\overline{a}b+\overline{b}c\right\vert m\sqrt{t}\sqrt
{1-t}\left\Vert f_{m}\right\Vert _{a}^{2}.
\]
Thus $A_{m}\geq m\left\Vert f_{m}\right\Vert _{a}^{2} f\left(  t\right)  $,
where
\[
f\left(  t\right)  =\left(  4\left\vert a\right\vert ^{2}+\left\vert
b\right\vert ^{2}\right)  t+\left(  \left\vert b\right\vert
^{2}+4\left\vert c\right\vert ^{2}\right)  \left(  1-t\right)
-4\left\vert \overline{a}b+\overline{b}c\right\vert \sqrt
{t}\sqrt{1-t}.
\]
Letting $C := \min \{f(t) : t\in [0,1]\}$, since  $A_{m}\geq m\left\Vert f_{m}\right\Vert _{a}^{2}f\left(  t\right)  $, for the Khavinson-Shapiro bounds to be satisfied it is enough  that  $f$ be strictly positive on $\left[
0,1\right].$

Recall our assumption $4ab\neq b^{2}.$ If $b=0$, then $ac\neq0,$ and 
\[
f\left(  t\right)  =4\left\vert a\right\vert ^{2}t+4\left\vert
c\right\vert ^{2}\left(  1-t\right)  .
\]
Since $f$ is affine and  $f\left(
0\right)  >0$, $f\left(  1\right)  >0$, we conclude that $C > 0$.

Assume next that $b\neq0$. Then
$$
f\left(  t\right)  =\left(  \left\vert b\right\vert ^{2}+4\left\vert
c\right\vert ^{2}\right)  \left(  \left(  1-t\right)  -\frac{4\left\vert
\overline{a}b+\overline{b}c\right\vert }{\left\vert
b\right\vert ^{2}+4\left\vert c\right\vert ^{2}}\sqrt{t}\sqrt
{1-t}+\frac{4\left\vert a\right\vert ^{2}+\left\vert b\right\vert
^{2}}{\left\vert b\right\vert ^{2}+4\left\vert c\right\vert ^{2}%
}t\right) 
$$
$$
=
\left(  \left\vert b\right\vert ^{2}+4\left\vert c\right\vert
^{2}\right) \times 
$$
$$
\left(  \left(  \sqrt{1-t}-\frac{2\left\vert \overline{a
	}b+\overline{b}c\right\vert }{\left\vert b\right\vert
	^{2}+4\left\vert c\right\vert ^{2}}\sqrt{t}\right)^{2} + 
\left(\frac{4\left\vert a\right\vert ^{2}+\left\vert
	b\right\vert ^{2}}{\left\vert b \right\vert ^{2}+4\left\vert
	c \right\vert ^{2}} - \frac
{4\left\vert \overline{a}b+\overline{b} c \right\vert ^{2}
\ }{\left(  \left\vert b\right\vert ^{2}+4\left\vert c\right\vert
^{2}\right)  ^{2}}\right) t \right).
$$
Note that $f\left(  0\right)  >0.$ Thus $f\left(  t\right)  >0$ for all
$t\in\left[  0,1\right]  $, provided the numerator obtained when summing the last two fractions is strictly positive, that is,
\[
B=\left(  4\left\vert a\right\vert ^{2}+\left\vert b\right\vert
^{2}\right)  \left(  \left\vert b\right\vert ^{2}+4\left\vert
c\right\vert ^{2}\right)  -4\left\vert \overline{a}b
+\overline{b}c\right\vert ^{2} > 0.
\]
Now
\[
\left\vert \overline{a}b+\overline{b}c\right\vert
^{2}=\left(  \overline{a}b+\overline{b}c\right)  \left(
a\overline{b}+b\overline{c}\right)  =\left\vert
a\right\vert ^{2}\left\vert b\right\vert ^{2}+\overline{a}
bb\overline{c}+\overline{b}c a\overline{b
}+\left\vert b^{2}\right\vert \left\vert c^{2}\right\vert,
\]
 so
\[
B=\ 16\left\vert ac\right\vert ^{2}+\left\vert b\right\vert
^{4}- 8\operatorname{Re}\left(  \overline{a}bb\overline{c
}\right).
\]

Write $z=b^{2}$ and $w=\overline{a}\overline{c}$, with 
$z=u+iv$ and $w=x+i y.$ Then
\[
B=16\left(  z^{2}+y^{2}\right)  +u^{2}+v^{2}- 8\left(  u x-v y\right)  =\left(
4x-u\right)^{2}+\left(  4y +v\right)^{2}.
\]
Hence $B=0$ entails that $y=-v/4$ and $x=u/4,$ so
\[
\overline{a}\overline{c}=w=x+iy= \frac{1}{4}\left(
u-iv\right)  =\frac{1}{4}\overline{z}=\frac{1}{4}\overline{b}^{2}.
\]
Thus $B=0$ if and only if $4ac=b^{2}.$ Since we assume that
$b^{2}\neq4ab$ we see that $A_{m}\geq  C m\left\Vert f_{m}\right\Vert _{a}^{2}.$

For $b)\implies c)$,  write $P(z_1, z_2) = r^2 z_1^2 + b z_1 z_2 + s^2 z_2^2$, and assume that  $b^2 = 4 r^2 s^2$. Then $P(z_1, z_2) = (r z_1 + s z_2)^2$.  
Finally, for $c)\implies a)$, given the nontrivial polynomial $P(z_1, z_2) = (r z_1 + s z_2)^2$, choose the nonzero vector $c =  (s, -r)$ and let $f_m(z_1, z_2) := (s z_1   - r z_2)^m$.   Then  
the equality $
\|P f_m \|_a^2 
= 
\|P \|_a^2 \|f_m\|_a^2
$   is a special case of Proposition \ref{lineal}.
\end{proof}

\section{Special cases: when $k= 1$ or $d=1$ the conjecture is true}

In view of the fact that the Khavinson-Shapiro bounds never hold when $k=1$, this case requires separate treatment. It is noted next that
 the Khavinson-Shapiro bounds are not
needed for lineal polynomials.

\begin{theorem} \label{linealconj}
Let $P_{1} (z_1, \dots , z_d)$ be a non-zero homogeneous polynomial of degree $1$, and let $P= P_1 - P_0$. Then the Fischer operator $F$ defined by
\[
F\left(  q\right)  := P_1^{\ast}\left(  D\right)  \left[  \left(  P_1- P_0\right)
q\right]
\]
is a bijection on the set of all entire functions in $d$ variables, i.e., $P = P_1- P_0$ and $P_1^{\ast
}\left(  D\right)  $ form a Fischer pair for this space.
\end{theorem}

\begin{proof}
We prove that $P_1 - P_0$ and $P_1^{\ast}\left(  D\right)  $ form a Fischer pair. By Shapiro's theorem \cite[Theorem 1]{Shap89}, $P_1$ and $P_1^{\ast}\left(  D\right)  $ form a Fischer pair for the entire functions. Since
$P_1\left(  z \right)  $ is homogeneous of degree $1$, for some $b \ne 0$ we have 
\[ 
P_1\left(  z\right)  =\left\langle z,b\right\rangle .
\]
Since $b \ne 0$, the linear funtion 
$\left\langle \cdot,b\right\rangle : \C^d \to \C$ is surjective, so
 there is a $z_{0}$ with $\left\langle z_{0}, b\right\rangle = P_0$,
and hence $P_1\left(  z-z_{0}\right)  =P_1\left(  z\right)  -  P_0.$

Let  $f$ be  entire, and set  $F\left(  z\right)  =f\left(
z + z_{0}\right)  .$  By Shapiro's theorem  there are unique entire functions $q$ and $h$ such that $P_1^{\ast}\left(  D\right)  h=0$ and 
\[
F\left(  z\right)  =P_1\left(  z\right)  q\left(  z\right)  +h\left(  z\right). 
\]
 Replace $z$ with
$z - z_{0}.$ Then
\begin{align*}
f\left(  z \right)   &  = F\left(  z -z_{0}\right)  =P_1\left(  z -z_{0}\right)
q\left(  z -z_{0}\right)  +h\left(  z-z_{0}\right) \\
&  =\left(  P_1\left(   z \right)  - P_0 \right)  q\left(  z -z_{0}\right)  +h\left(
z-z_{0}\right)  .
\end{align*}
Clearly $\widetilde{h}\left(  z \right)  : =h\left(  z - z_{0}\right)  $ satisfies
$P_1^{\ast}\left(  D\right)  \widetilde{h}=0$, since $P_1^{\ast}\left(  D\right)
h=0.$ Uniqueness of $\widetilde{h}\left(  z \right)$ and $\widetilde{q}\left(  z \right)  : =q\left(  z - z_{0}\right)  $
follows from the corresponding uniqueness statements for $h$ and $q$.
\end{proof}

\vskip .3 cm

For $d=1$,   the conjecture  can be proven directly:

\begin{theorem} \label{dim1}
	Let $P:\C \to \C$ be a non-zero  polynomial of degree $k$,
	with  homogeneous principal part $P_{k}$. Then  $P$ and $P_k^{\ast
	}\left(  D\right)  $ form a Fischer pair for $E(\C)$. Furthermore, if $f$ has order $\rho$, writing $f = P q + r$ we find that $q$ also has order $\rho$ and $r$ is a polynomial of degree $k -1$. 
\end{theorem}

\begin{proof} Let $P$ be a non-zero polynomial of degree $k$ with  complex  zeros $\alpha
	_{1},\dots ,\alpha_{k}$, listed according to their multiplicity. Given an arbitrary entire function $f$, define $I\left(  f\right)  $ to be the unique
 polynomial of degree $k-1$ interpolating $f$ at the zeros of $P$ (using the Lagrange interpolation polynomial if all the zeros are different, or more generally, using Hermite interpolation in the case of repeated roots, so not only the values of $f$ and $I(f)$ coincide at the roots, but also an appropriate number of derivatives do so).
 Then $f-I\left(  f\right)  $ vanishes at the
	points $\alpha_{1},\dots,\alpha_{k}$ and we can write $f-I\left(  f\right)  =Pq$
	for some entire function $q.$ This yields the Fischer decomposition, because trivially $P_k^{\ast
	}\left(  D\right) (I(f)) = 0 $. Uniqueness follows from the uniqueness of the interpolating polynomial, since given any decomposition $f = Pq^\prime + r^\prime$, in order for $P_k^{\ast
}\left(  D\right) (r^\prime) = 0 $ to hold, $r^\prime$ must be a polynomial of
degree strictly smaller than $k$.

 When $f$ is
	entire of order $\rho$, then both $f-I\left(
	f\right)  $ and $q$ also have order $\rho$.
\end{proof}

\vskip .2cm 

We recall next some additional results regarding the conjecture. It  also holds when $P_{k}^{\ast}\left(  z\right)  $
is of the form $z_{1}^{k}$ for $z=\left(  z_{1},z^{\prime}\right)
\in\mathbb{C}\times\mathbb{C}^{d-1}.$ In 	\cite{MeSt85} Meril and Struppa have proven the
following result: 

\begin{theorem}
	Let $P\left(  z\right)  $ be a polynomial of degree $k,$  let $Q_{k}\left(
	z\right)  =z_{1}^{k}$, where   $z=\left(  z_{1},z^{\prime}\right)  \in
	\mathbb{C}\times\mathbb{C}^{d-1}$, let $C\neq0$ be a complex number, and let $p_{0},\dots,p_{k-1}$ be polynomials in
	the variable $z^{\prime}\in\mathbb{C}^{d-1}$. Then the polynomial $P$ and the
	differential operator  $Q_{k}\left(  D\right)  =\partial_{1}^{k}$ form a
	Fischer pair if and only if $P$ is of the form
	\[
	P\left(  z\right)  =Cz_{1}^{k}+p_{k-1}\left(  z^{\prime}\right)  z_{1}%
	^{k-1}+\cdots+p_{0}\left(  z^{\prime}\right).
	\]
\end{theorem}

Finally we mention that the following conjecture: 

\begin{conjecture}
[II]Let $P$ be a polynomial. Then the Fischer operator $F_{P}:E\left(
\mathbb{C}^{d}\right)  \rightarrow E\left(  \mathbb{C}^{d}\right)  $ defined
by
\[
F_{P}\left(  q\right)  =P^{\ast}\left(  D\right)  \left(  Pq\right)
\]
is a bijection. 
\end{conjecture}

A. Meril and A. Yger  have shown that $F_{p}$ is injective
when $P$ is a polynomial of degree $\leq2$, cf. \cite{MeYg92}. In dimension $d=2$, they have
also proven that the Fischer operator $F_{P}:E\left(  \mathbb{C}^{2}\right)
\rightarrow E\left(  \mathbb{C}^{2}\right)  $ is bijective for any polynomial
of degree $\leq2$. In \cite{ElNa12} it is shown that conjecture II holds for the
polynomial $P\left(  z\right)  =1+z^{\alpha}$, where $\alpha\in\mathbb{N}
_{0}^{d}$ has only positive entries. In general,  conjecture II is still
open.

\bigskip 

For more information about  Fischer operators and their  relationship to problems in 
analysis we refer to \cite{LuRe11}, \cite{Rend16} and the classical paper \cite{Shap89}.

\section{Proof of Theorem \ref{main}}

Let us start with some preliminary bounds.

\begin{lemma}
	Given a multi-index $\alpha\in\mathbb{N}_{0}^{d}$, the estimate
	\begin{equation}
		\left\Vert f_{m}\right\Vert _{a}\leq\left\Vert z^{\alpha}f_{m}\right\Vert
		_{a}\leq C_{\alpha,m}\left\Vert f_{m}\right\Vert _{a} \label{eq1N}
	\end{equation}
	holds for all homogeneous polynomials $f_{m}$ of degree $m$, where
	\[
	C_{\alpha,m}=\sup_{\beta\in\mathbb{N}_{0}^{d},\left\vert \beta\right\vert
		=m}\sqrt{\frac{\left(  \alpha+\beta\right)  !}{\beta!}},
	\]
	and this is the smallest constant such that (\ref{eq1N}) holds for all
	homogeneous polynomials of degree $m.$
\end{lemma}

\begin{proof}
	We consider $f_{m} (z) = z^{\beta}$ with $\left\vert \beta\right\vert =m.$ Then
	\[
	\left\Vert z^{\alpha}f_{m}(z)\right\Vert _{a}^{2}=\left\Vert z^{\alpha+\beta
	}\right\Vert _{a}^{2}=\left(  \alpha+\beta\right)  !=\frac{\left(
		\alpha+\beta\right)  !}{\beta!}\left\Vert z^{\beta}\right\Vert _{a}^{2}
	=\frac{\left(  \alpha+\beta\right)  !}{\beta!}\left\Vert f_{m}\right\Vert
	_{a}^{2}.
	\]
	It follows that
	\[
	C_{\alpha,m}\geq\sup_{\beta\in\mathbb{N}_{0}^{d},\left\vert \beta\right\vert
		=m}\sqrt{\frac{\left(  \alpha+\beta\right)  !}{\beta!}}.
	\]
	Now we show that $C_{\alpha,m}$ is a suitable constant. Let us write 
	$f_{m} (z) =
	{\displaystyle\sum_{\left\vert \beta\right\vert =m}}
	a_{\beta}z^{\beta}.$ Then $z^{\alpha}f_{m}(z) =
	{\displaystyle\sum_{\left\vert \beta\right\vert =m}}
	a_{\beta} z^{\alpha+\beta}$ and
	\begin{align*}
	\left\Vert f_{m}\right\Vert _{a}^2
	\leq
		\left\Vert z^{\alpha}f_{m}(z) \right\Vert _{a}^{2}  &  =
		{\displaystyle\sum_{\left\vert \beta\right\vert =m}}
		\left(  \alpha+\beta\right)  !\left\vert a_{\beta}\right\vert ^{2}=
		{\displaystyle\sum_{\left\vert \beta\right\vert =m}}
		\frac{\left(  \alpha+\beta\right)  !}{\beta!}\beta!\left\vert a_{\beta
		}\right\vert ^{2}\\
		&  \leq
		{\displaystyle\sum_{\left\vert \beta\right\vert =m}}
		C_{\alpha,m}^2\beta!\left\vert a_{\beta}\right\vert ^{2}=C_{\alpha,m}
		^{2}\left\Vert f_{m}\right\Vert _{a}^{2}.
	\end{align*}
	\end{proof}

With different normalizations, the following result is essentially due to B. Beauzamy, cf. \cite[Formula (6)]{Beau97}.

\begin{lemma} \label{lemmaD}
	If $P\left(
	z \right)  = 
	{\displaystyle\sum_{\left\vert a\right\vert =k}}
	c_{\alpha}z^{\alpha}$ is a homogeneous polynomial of degree $k$, 
	then
	\[
	\left\Vert Pf_{m}\right\Vert _{a}\leq \left\Vert f_{m}\right\Vert
	_{a}\left(  1+m\right)  ^{\frac{k}{2}}
	{\displaystyle\sum_{\left\vert \alpha\right\vert =k}}
	\left\vert c_{\alpha}\right\vert \sqrt{\alpha!}.
	\]
\end{lemma}

\begin{proof}
	By the preceding lemma, 
	\[
	\left\Vert Pf_{m}\right\Vert _{a}\leq {\displaystyle\sum_{\left\vert a\right\vert =k}}
	\left\vert c_{\alpha}\right\vert \left\Vert z^{\alpha}f_{m}\right\Vert
	_{a}\leq\left\Vert f_{m}\right\Vert _{a}
	{\displaystyle\sum_{\left\vert a\right\vert =k}}
	\left\vert c_{\alpha}\right\vert \sup_{\beta\in\mathbb{N}_{0}^{d},\left\vert
		\beta\right\vert =m}\sqrt{\frac{\left(  \alpha+\beta\right)  !}{\beta!}}.
	\]
	Note next that
	\begin{align*}
		\frac{\left(  \alpha+\beta\right)  !}{\alpha!\beta!}  &  =\frac{\left(
			\alpha_{1}+\beta_{1}\right)  !\cdots \left(  \alpha_{d}+\beta_{d}\right)
			!}{\alpha_{1}!\cdots \alpha_{d}!\beta_{1}!\cdots \beta_{d}!}\\
		&  =\frac{\left(  \beta_{1}+1\right)  \left(  \beta_{1}+2\right)
			\cdots \left(  \beta_{1}+\alpha_{1}\right)  }{\alpha_{1}!}
		\cdots 
		\frac{\left(  \beta_{d}+1\right)  \cdots \left(  \beta_{d}%
			+\alpha_{d}\right)  }{\alpha_{d}!} \\
		&  \leq\left(  1+\beta_{1}\right)  ^{\alpha_{1}}\cdots \left(  1+\beta
		_{d}\right)  ^{\alpha_{d}}\leq\left(  1+m\right)  ^{\left\vert \alpha
			\right\vert }.
	\end{align*}
	\end{proof}

We will use the following result,  which appears in \cite[Theorem 17]{Rend08}. More details regarding its proof are presented in \cite[Theorem 7]{RendAl22}. While in the latter reference the result is stated for even degrees $2k$, the argument presented there works for
general values of $k$.

\begin{theorem}
	Let $Q$ be a homogeneous polynomial of degree $k > 0$, let $P$ be a
	polynomial of degree $k$ of the form 
	\begin{equation}
P=P_{k}-P_{k-1}- \cdots -P_{0},\   \label{eqHom}
\end{equation} 
and assume that $(P_{k},Q)$ is a Fischer pair for $\mathcal{P}\left( \mathbb{C}^{d}\right) $. Setting $T:=T_{P_{k}}$, we have 
	\begin{equation}
		T_{P}\left( f_{m}\right)
		=\sum_{j=-1}^{m}\sum_{s_{0}=0}^{k-1}\sum_{s_{1}=0}^{k-1} \dots
		\sum_{s_{j}=0}^{k-1}T(P_{s_{j}}T(\cdots P_{s_{1}}
		T(P_{s_{0}}T(f_{m}))\cdots))  \label{eqinduct}
	\end{equation}
	for all homogeneous polynomials $f_{m}$ of degree $m$, with the convention
	that the summand for $j=-1$ is $Tf_{m}$.
\end{theorem}

 We have noted above that when the remainder $r = 0$ we have  $f_m = P_k Tf_m$, and thus assuming the Khavinson-Shapiro bounds is equivalent to the condition 
\begin{equation}\label{kTfbis}
	\left\Vert Tf_{m}\right\Vert _{a}^2\leq\frac{C}{m^{\tau}}\left\Vert
	f_{m}\right\Vert _{a}^2
\end{equation}
actually used in the proof of Theorem \ref{main}. However, it is conceivable that for a particular pair $(f, P)$
and all the homogeneous polynomials appearing in (\ref{eqinduct}), bounds of type (\ref{kTfbis}) might actually be strictly weaker than the Khavinson-Shapiro conditions. But we will not pursue these elaborations here.

To simplify expressions such as (\ref{eqinduct})  parentheses shall often be omitted.
We will also use the following result of H.S. Shapiro \cite[p. 519]{Shap89}:

\begin{lemma}
\label{LemShapiro}(H.S. Shapiro) Suppose that $f_{k} (z)=\sum_{\left\vert
\alpha\right\vert =k}c_{\alpha}z^{\alpha}$ is a homogeneous polynomial of
degree $k$. Then for every complex vector $z\in\mathbb{C}^{d}$ the following
estimate holds:
\[
\left\vert f_{k}\left(  z\right)  \right\vert ^{2}\leq\frac{1}{k!}\left\vert
z\right\vert ^{2k}\left\Vert f_{k}\right\Vert _{a}^{2}.
\]

\end{lemma}

Let us recall some well known definitions and facts about entire functions (additional details and references can be found in \cite{RendAl22}, cf. also \cite{Armi04}). 
The order $\rho _{\mathbb{C}^{d}}\left( f\right) $ of a continuous function $f:
\mathbb{C}^{d}\rightarrow \mathbb{C}$ is defined by setting 
\begin{equation*}
M_{\mathbb{C}^{d}}\left( f,r\right) :=\sup \left\{ \left\vert f\left(
z\right) \right\vert :z\in \mathbb{C}^{d},\left\vert z\right\vert =r\right\},
  \label{eqM1}
\end{equation*}
and then   
\begin{equation*}
\rho _{\mathbb{C}^{d}}\left( f\right) :=\lim_{r\rightarrow \infty }\sup \frac{%
\log \log M_{\mathbb{C}^{d}}\left( f,r\right) }{\log r}\in \left[ 0,\infty %
\right] .
\end{equation*}

Given  an entire function $f$, we write $f=\sum_{m=0}^{\infty}f_{m}$, where the
homogeneous polynomials $f_{m}$ are those given by the Taylor expansion about 0,  that is,
\begin{equation} \label{taylor}
f_{m}\left( z \right) =\sum_{\left\vert \alpha \right\vert =m}\frac{1}{\alpha
!} \ \partial^{\alpha }f \left( 0\right)
\; z^{\alpha }\text{ for }m\in \mathbb{N}_{0}.
\end{equation}
It is well known that for $\rho \ge 0$, we have $\rho _{
\mathbb{C}^{d}}\left( f\right) \leq \rho $ if and only if for every 
$
\varepsilon >0$, there exists an $m_0  \geq 0$ such that for every $m \ge m_0$ the following bounds hold:
\begin{equation}
\max_{\theta \in \mathbb{S}^{d-1}}\left\vert f_{m}\left( \theta \right)
\right\vert \leq \frac{1}{m^{m/\left( \rho +\varepsilon \right) }}.
\label{eqtaylor}
\end{equation}

We shall also use an old  result due to  V. Bargmann, cf. \cite{Bar}.

\begin{theorem}
Let $P$ and $Q$ be polynomials in $d$ complex variables. Then
\begin{equation}
\left\langle P,Q\right\rangle _{a}=\frac{1}{\pi^{d}}\int_{\mathbb{R}^{d}}
\int_{\mathbb{R}^{d}}P\left(  x+iy\right)  \overline{Q\left(  x+iy\right)
}e^{-\left|  x\right|  ^{2}-\left|  y\right|  ^{2}}dxdy<\infty\label{Bargman},
\end{equation}
where $dxdy$ is Lebesgue measure on $\mathbb{R}^{2d}$.
\end{theorem}

\begin{lemma} \label{Bargma}
Let $f_{m}$ be a homogeneous polynomial of degree $m$, and denote by $
\mathbb{S}^{2d-1}$ the unit sphere
 in $\mathbb{R}^{2n}$. There is a dimensional constant $C_d > 0$ such that, identifying
 $\mathbb{C}^{n}$ with $\mathbb{R}^{2n}$ as a measure space, we have
\[
\left\Vert f_{m}\right\Vert _{a}
\leq
C_d \sqrt{\left(
m+d-1\right)  !}\max_{\theta\in\mathbb{S}^{2d-1}}\left\vert f_{m}\left(
\eta\right)  \right\vert .
\]
\end{lemma}

\begin{proof}
 Let $f_{m}:\mathbb{C}^{d}\rightarrow\mathbb{C}$ be a
homogeneous polynomial of degree $m$, and recall that for $x > 0$, the Gamma function is defined as $\Gamma (x) := \int_{0}^{\infty}e^{-t} \ t^{x - 1}dt$. 
Integrating in polar coordinates and using the
change of variables $t = r^2$ we get
\[
\left\Vert f_{m}\right\Vert _{a}^{2}
=
\frac{1}{\pi^{d}}\int_{\mathbb{C}^{d}
}\left\vert f_{m}\left(  z\right)  \right\vert ^{2}e^{-\left\vert z\right\vert
^{2}}dz=
\frac{1}{\pi^{d}}\int_{0}^{\infty}e^{-r^{2}}r^{2m+2d-1}dr\int
_{\mathbb{S}^{2d-1}}\left\vert f_{m}\left(  \eta\right)  \right\vert
^{2}d\eta
\]
\[
\le
C_d \ \Gamma (m + d) \max_{\theta\in\mathbb{S}^{2d-1}}\left\vert f_{m}\left(
\eta\right)  \right\vert.
\]
\end{proof}

Next we present the proof of our main result, Theorem  \ref{main}.

\vskip .2 cm

\begin{proof} In view of Theorem \ref{linealconj} we may suppose that $k\ge 2$.

Set $P = P_{k}-P_{\beta}- \cdots  -P_{0}$. 
Then $T_{P}\left(
f_{m}\right)  $ is either the zero polynomial or a polynomial of degree $<m$
(not necessarily homogeneous). Our strategy is to show that
\begin{equation}
g:=
{\displaystyle\sum_{m=0}^{\infty}}
T_{P}\left(  f_{m}\right)  \label{eq16}
\end{equation}
 defines an entire
function $g:\mathbb{C}^{d}\rightarrow\mathbb{C}$ of order bounded by $\rho$,
 by writing $g\left(  z\right)  =
{\displaystyle\sum_{M=0}^{\infty}}
G_{M}\left(  z\right)  ,$ where each $G_{M}$ is a homogeneous polynomial of
degree $M$, and then applying the criterion presented in (\ref{eqtaylor}). As a reminder, we mention that when the representation of $g$ as $\sum_{M=0}^{\infty}
G_{M}$ exists, it is
unique. 

It has been noted above that
\begin{equation}
T_{P}\left(  f_{m}\right)  =\sum_{j=-1}^{m}\sum_{s_{0}=0}^{k-1}\sum_{s_{1}%
=0}^{k-1}\cdots \sum_{s_{j}=0}^{k-1}TP_{s_{j}}\cdots  TP_{s_{0}}Tf_{m}, \label{eq17}%
\end{equation}
and clearly, $TP_{s_{j}}\cdots TP_{s_{0}}Tf_{m}$ is a homogeneous polynomial of
degree
\[
[s_{0}+\cdots +s_{j}+m-k\left(  j+2\right)]_+  .
\]
If $s_{0},\dots ,s_{n}\in\left\{  0,\dots , \beta \right\}  $ are given, where by definition $\beta \le k - 1$,  and if $n>m,$
then $TP_{s_{n}}\cdots TP_{s_{0}}Tf_{m}$ is zero by inspection of its degree:
\begin{align*}
m+s_{0}+\cdots +s_{n}-k\left(  n+2\right)   &  \leq m+\left(  n+1\right)
\left(  k-1\right)  -k\left(  n+2\right) \\
&  =m-k-n-1<0,
\end{align*}
so the $m$ in the first summatory of (\ref{eq17}) can be replaced by $\infty.$
By hypothesis, $P_{s}=0$ for all $s\in\left\{  \beta +1,\dots ,k-1\right\}  ,$ with the convention that this set is presented in increasing order, so it is empty when $\beta = k - 1$. Hence, 
we can write
\[
T_{P}\left(  f_{m}\right)  =\sum_{j=-1}^{\infty}\sum_{s_{0}=0}^{\beta}
\sum_{s_{1}=0}^{\beta}\cdots \sum_{s_{j}=0}^{\beta}TP_{s_{j}}\cdots TP_{s_{0}}Tf_{m}.
\]
In order to show that the sum in (\ref{eq16}) defines an entire function it suffices to prove that
\[
G=\sum_{j=-1}^{\infty}\sum_{s_{0}=0}^{\beta}\sum_{s_{1}=0}^{\beta}
\cdots \sum_{s_{j}=0}^{\beta}
{\displaystyle\sum_{m=0}^{\infty}}
TP_{s_{j}}\cdots TP_{s_{0}}Tf_{m}
\]
does so. Then the sum can be reordered and shown to be equal to $g.$ Next we
collect all summands having degree $M\geq0.$ The requirement
\[
\deg TP_{s_{j}}\cdots TP_{s_{0}}Tf_{m}=s_{0}+\cdots +s_{j}+m-k\left(  j+2\right)
=M
\]
means that $m=M+k\left(  j+2\right)  -\left(  s_{0}+\cdots +s_{j}\right)  ,$
and therefore we consider the sum

\begin{equation}\label{gm}
G_{M} (z) :=\sum_{j=-1}^{\infty}\sum_{s_{0}=0}^{\beta}\sum_{s_{1}=0}^{\beta}%
\cdots \sum_{s_{j}=0}^{\beta}
TP_{s_{j}}\cdots TP_{s_{0}}Tf_{M+k\left(  j+2\right)  -\left(  s_{0}+\cdots 
+s_{j}\right)  } (z).
\end{equation}
Next we show that $G_{M}$ converges absolutely everywhere.
Note that while $G_M$ will turn out to be a homogeneous polynomial of degree $M$, contributions to it may come from infinitely many values of $j$, so a proof of convergence is required. 

By Lemma \ref{LemShapiro}, for any complex vector $z\in\mathbb{C}^{d}$ and any
homogeneous polynomial $h_{M}$ of degree $M$ the following estimate holds:
\[
\left\vert h_{M}\left(  z\right)  \right\vert \leq\frac{1}{\sqrt{M!}
}\left\vert z\right\vert ^{M}\left\Vert h_{M}\right\Vert _{a}.
\]
Thus we have
\begin{equation} 	\label{ModApo}
\left\vert TP_{s_{j}}\cdots TP_{s_{0}}Tf_{M+k\left(  j+2\right)  -\left(
s_{0}+\cdots +s_{j}\right)  }\left(  z\right)  \right\vert 
\end{equation}
\begin{equation} 	\label{ModApo1}
 \leq\frac{\left\vert z\right\vert ^{M}}{\sqrt{M!}}\left\Vert TP_{s_{j}
}\cdots TP_{s_{0}}Tf_{M+k\left(  j+2\right)  -\left(  s_{0}+\cdots +s_{j}\right)
}\right\Vert _{a}.
\end{equation}
Since
\[
\left\Vert Tf_{m}\right\Vert _{a}\leq\frac{C}{m^{\tau/2}}\left\Vert
f_{m}\right\Vert _{a}
\]
and $P_{s_{j}}\cdots TP_{s_{0}}Tf_{M+k\left(  j+2\right)  -\left(  s_{0}
+\cdots +s_{j}\right)  }$ has degree
$
M+k$,
it follows that
\begin{equation}\label{sjm}
S_{j,M}  :=\left\Vert TP_{s_{j}}\cdots TP_{s_{0}}Tf_{M+k\left(  j+2\right)  -\left(
s_{0}+\cdots +s_{j}\right)  }\right\Vert _{a}
\end{equation}
\begin{equation}
\leq
\frac{C}{\left(  M+k\right)  ^{\tau/2}}\left\Vert P_{s_{j}
}\cdots TP_{s_{0}}Tf_{M+k\left(  j+2\right)  -\left(  s_{0}+\cdots +s_{j}\right)
}\right\Vert _{a}
\end{equation}
(note that $S_{j,M}$ depends also on $\left(  s_{0}+\cdots +s_{j}\right)$, but we omit this fact from the notation).
Now
\[
TP_{s_{j-1}}\cdots TP_{s_{0}}Tf_{M+k\left(  j+2\right)  -\left(  s_{0}
+\cdots +s_{j}\right)  }
\]
has degree
$
M+k-s_{j},
$
so using the abbreviation 
$$
D_{P}:=
{\displaystyle\sum_{\left\vert \alpha\right\vert =k}}
\left\vert c_{\alpha}\right\vert \sqrt{\alpha!},
$$
and, to reduce the number of subindices, writing 
$D_{s_{j}}:=D_{  P_{s_{j}}} $, from Lemma \ref{lemmaD}  we get
\[
\left\Vert P_{s_{j}}f_{m}\right\Vert _{a}\leq
D_{s_{j}} \left\Vert f_{m}\right\Vert _{a} \left(  m+1\right)  ^{s_{j}/2}.
\]
Hence we see
that
\[
S_{j,M} 
\leq CD_{s_{j}}\frac{\left(  M+k-s_{j}+1\right)  ^{s_{j}/2}}{\left(
M+k\right)  ^{\tau/2}}\left\Vert TP_{s_{j-1}}\cdots TP_{s_{0}}Tf_{M+k\left(
j+2\right)  -\left(  s_{0}+\cdots +s_{j}\right)  }\right\Vert _{a}.
\]
Proceed inductively to obtain positive numbers $A_{j}
,\dots, A_{0}$ and $B_{j},\dots ,B_{-1}$ such that
	\[
S_{j,M} \leq\frac{C^{j+1}D_{s_{j}}\cdots D_{s_{0}}A_{j}^{s_{j}/2}A_{j-1}^{s_{j-1}
		/2}\cdots A_{0}^{s_{0}/2}\left\Vert f_{M+k\left(  j+2\right)  -\left(
		s_{0}+\cdots+s_{j}\right)  }\right\Vert _{a}}{B_{j}^{\tau/2}B_{j-1}^{\tau
		/2}\cdots B_{-1}^{\tau/2}},
\]
where
$B_{j}=M+k$, 
$
B_{n-1}=B_{n}+k-s_{n}
$
 for  $n =j,j-1,\dots , 0$, and
\[
A_{n}= B_{n}-s_{n}+1  =  B_{n - 1}- k + 1 \text{ for } n =j,j-1,\dots, 0.
\]
Thus 
$B_{j-1}   = B_{j}+k-s_{j} = M+2k-s_{j}$, and in general we have 
\[
B_{j-n} = M+\left(  n +1\right)  k-\left(  s_{j}+s_{j-1}+\cdots
+s_{j-n+1}\right);
\]
note that the largest term is the last one:
\[
B_{-1} = M+\left(  j+2\right)  k-\left(  s_{j}+s_{j-1}+\cdots
+s_{0}\right) = m.
\]
Now $A_{n} < B_{n - 1}$,
	so  $A_{n}^{s_n}\leq B_{n - 1}^{s_{n}}$ for all $s_n$, where $0 \le s_{n}\leq \beta$.
	It follows that
	\[
	S_{j,M} 
	\leq C^{j+1}D_{s_{j}}\cdots D_{s_{0}} B_{j-1}^{(s_{j}  -\tau)/2}\cdots B_{-1}^{(s_{0}  -\tau) /2}\left\Vert f_{M+k\left(  j+2\right)  - \sum_{n=0}^{j } s_{n}  }\right\Vert_{a}
	\]
	\[
	\leq
	C^{j+1}D_{s_{j}}\cdots D_{s_{0}}  B_{-1}^{2^{-1} \sum_{n=0}^{j } (s_{n}  -\tau)}\left\Vert f_{M+k\left(  j+2\right)  -\sum_{n=0}^{j } s_{n}  }\right\Vert_{a}.
	\]
By Lemma \ref{Bargma}, there is a dimensional constant $C_d$ such that
	\[
	\left\Vert f_{m}\right\Vert _{a}
	\le
	C_d \ \sqrt{\left( m+d-1\right) !} \max_{\eta\in\mathbb{S}^{2d-1}}\left\vert
	f_{m}\left(  \eta\right)  \right\vert,
	\]
	so we have
$$
		\left\Vert f_{M+k\left(  j+2\right)  -\sum_{n=0}^{j } s_{n}
		}\right\Vert _{a}  
		\leq
		C_d \ \sqrt{\left( B_{-1}  +d-1\right)  !} 
\max_{\eta\in\mathbb{S}^{2d-1}}\left\vert f_{B_{-1}}\left(  \eta\right)  \right\vert.
$$
Note that for every $M$ sufficiently large and all $j$, or for every $j$ sufficiently large
and all $M$,
$$
\frac{\left( B_{-1}  +d-1\right)  ! }{M  !}
\le
\left( B_{-1}  +d-1\right)^{ B_{-1} + d - 1 - M}
$$
$$
\le
\left( B_{-1}\right)^{ B_{-1} + d - 1 - M} 
\left( 1  + \frac{d-1}{B_{-1}}\right)^{ B_{-1} + d - 1 - M}
$$
$$
\le
e^{d^2} \left( B_{-1}\right)^{ B_{-1} + d - 1  - M},
$$
so
$$
\frac{ B_{-1}^{2^{-1} \sum_{n=0}^{j } (s_{n}  -\tau)} \sqrt{\left( B_{-1}  +d-1\right)  !} }{\sqrt{M  !}}
$$
$$
\le
B_{-1}^{2^{-1} \sum_{n=0}^{j } (s_{n}  -\tau)} e^{d^2}  B_{-1}^{ 2^{-1} \left(M+\left(  j+2\right)  k- \sum_{n=0}^{j } s_{n}  + d -1  - M \right)}
$$
$$
=
e^{d^2}  B_{-1}^{ 2^{-1} \left(k + d - 1 + \left(  j+1\right)  (k- \tau)  \right)},
$$
from whence it follows that
$$
\label{eqMAINNEW}
\frac{S_{j,M} }{\sqrt{M!}} 
		\leq
e^{d^2}  C_d \ C^{j+1}D_{s_{j}}\cdots D_{s_{0}} B_{-1}^{2^{-1} (k + d -1  + (j + 1)(k  -\tau))} 
 \max_{\eta\in\mathbb{S}^{2d-1}}\left\vert f_{B_{-1} }\left(  \eta\right)  \right\vert.
$$	
	Since $f$ has order $\rho$, the bound (\ref{eqtaylor}) entails that for every
	$\varepsilon>0$ sufficiently small there exists a constant $A_{\varepsilon}$
	such that
	\begin{equation}
		\max_{\eta\in\mathbb{S}^{2d-1}}\left\vert f_{m}\left(  \eta\right)
		\right\vert \leq\frac{A_{\varepsilon}}{m^{\frac{m}{\rho+\varepsilon}}}
		\label{eqaa}
	\end{equation}
	for all natural numbers $m.$ Using (\ref{eqaa}) and then replacing $B_{ -1 }^{- \frac{B_{ -1 } }{\rho+\varepsilon}}$ with the larger quantity
	$B_{ -1 }^{- \frac{ M + k + (j + 1)(k  -\beta))}{\rho+\varepsilon}}$, we get
	\begin{equation} 	\label{fixM}
\frac{S_{j,M}}{\sqrt{M!}} 
	\leq
	e^{d^2}   C_d A_{\varepsilon}\ C^{j+1}D_{s_{j}}\cdots D_{s_{0}} B_{-1}^{\frac{(k + d - 1 + (j + 1)(k  -\tau))}{2}} B_{ -1 }^{- \frac{B_{ -1 } }
		{\rho+\varepsilon}}
	\end{equation}
	\begin{equation}
	\label{fixM0}
	\le 
	e^{d^2}   C_d A_{\varepsilon}\ C^{j+1}D_{s_{j}}\cdots D_{s_{0}} B_{-1}^{\frac{(j + 1)(k  -\tau)}{2}}B_{ -1 }^{- \frac{(j + 1)(k  -\beta)}{\rho+\varepsilon}} B_{-1}^{\frac{(k + d - 1)}{2}} B_{ -1 }^{- \frac{ M + k}{\rho+\varepsilon}}.
	\end{equation}
From the hypothesis $\rho<\frac{2(k-\beta)}{k-\tau}
$, by selecting $\varepsilon$ small enough we conclude that
$$
\frac{k-\tau}{2} < \frac{k-\beta}{\rho+  \varepsilon}.
$$
Thus,  replacing $ B_{-1}$ in  $B_{-1}^{\frac{k  -\tau}{2} - \frac{k  -\beta}{\rho+\varepsilon}}$
by something smaller, namely
$$
M + k + (j + 1)(k - \beta),
$$
we get a larger quantity in (\ref{fixM}), to wit,
	\begin{equation}
	\label{fixM1}
\frac{S_{j,M}}{\sqrt{M!}} 
\leq 
e^{d^2}   C_d A_{\varepsilon}\ C^{j+1}D_{s_{j}}\cdots D_{s_{0}} 
\left(\left(M + k + (j + 1)(k - \beta)\right)^{\frac{k  -\tau}{2} - \frac{k  -\beta}{\rho+\varepsilon}} \right)^{j + 1} 
\end{equation}
	\begin{equation} 	\label{fixM2}
\times B_{-1}^{\frac{(k + d - 1)}{2}} B_{ -1 }^{- \frac{ M + k}{\rho+\varepsilon}}.
	\end{equation}
We can make the same substitution in $B_{-1}^{\frac{(k + d - 1)}{2}} B_{ -1 }^{- \frac{ M + k}{\rho+\varepsilon}}
$ when
$\frac{(k + d - 1)}{2} \le  \frac{ M + k}{\rho+\varepsilon},
$ 
and use $
B_{-1} \le M + (j + 2)k 
$ when
$\frac{(k + d - 1)}{2} >  \frac{ M + k}{\rho+\varepsilon}.
$ In either case, the $j + 1$-root of the 
corresponding quantity approaches 1 as $j \to \infty$, so for $M \in \N_0$ fixed
and $j$ sufficiently large we obtain
	\begin{equation}\label{fixM3}
B_{-1}^{\frac{(k + d - 1)}{2}} B_{ -1 }^{- \frac{ M + k}{\rho+\varepsilon}}
\le 2^{j + 1}
\end{equation}
for all choices of $s_0, \dots, s_j$ between $0$ and $\beta$.

\vskip .2cm

After these preliminary bounds, we prove that the sum in (\ref{gm}) is well defined and converges absolutely for every $z\in \C^d\setminus \{0\}$ and every natural number $M$. So fix $M$ and set $z = r \eta$, where $r > 0$ and $\eta \in \mathbb{S}^{2d-1}$. 
Let us write
\[
G_{M}^{\left(  j\right)  }\left(  r \eta\right)  := r^M \sum_{s_{0}=0}^{\beta}
\sum_{s_{1}=0}^{\beta}\cdots \sum_{s_{j}=0}^{\beta}\left\vert TP_{s_{j}}....TP_{s_{0}%
}Tf_{M+k\left(  j+2\right)  -\left(  s_{0}+\cdots+s_{j}\right)  }\left(
\eta\right)  \right\vert,
\]
and $\widetilde{D}:=D_{s_{0}}+\cdots+D_{s_{\beta}}.$ Note that we have
\[
\sum_{s_{0}=0}^{\beta}\sum_{s_{1}=0}^{\beta}\cdots\sum_{s_{j}=0}^{\beta}D_{s_{j}}\cdots
D_{s_{0}}=\left(  D_{s_{0}}+\cdots+D_{s_{\beta}}\right)  ^{j+1}=\widetilde
{D}^{j+1}.
\]
Putting together (\ref{fixM1})-(\ref{fixM2}) with the bound from (\ref{fixM3}), and using (\ref{ModApo})-(\ref{ModApo1}), we obtain
for sufficiently large values of $j$,
$$
G_{M}^{\left(  j\right)  }\left(  r \eta\right)  
\leq
r^M	e^{d^2}   C_d A_{\varepsilon}\cdot\left( 
(M + k + (j + 1)(k - \beta))^{\frac{k - \tau}{2} - \frac{k - \beta} {\rho+\varepsilon}} 2 C \widetilde{D} \right)
^{j+1}
$$
$$
\leq
\left(( 
(M + k + (j + 1)(k - \beta))^{\frac{k - \tau}{2} - \frac{k - \beta} {\rho+\varepsilon}} 3 C \widetilde{D} \right)
^{j+1}.
$$
Hence,  we can find a natural number $j_{0}$ such that for all
$j\geq j_{0}$ the last inequality is satisfied and furthermore
\[
(M + k + (j + 1)(k - \beta))^{\frac{k - \tau}{2} - \frac{k - \beta} {\rho+\varepsilon}} 3 C \widetilde{D}
\leq
\frac{1}{2}.
\]
It follows that $
{\displaystyle\sum_{j=0}^{\infty}}
G_{M}^{\left(  j\right)  }$ converges on $\C^d$, so $G_{M}$ is a well-defined homogeneous polynomial for each
$M\in\mathbb{N}_{0}.$

Next we show that for sufficiently large values of $M$, the polynomial $G_M$ on the unit sphere satisfies the bounds given in (\ref{eqtaylor}). Recall that
$$
|G_{M}(\eta)| 
\le
{\displaystyle\sum_{j=0}^{\infty}}
G_{M}^{\left(  j\right)}   (\eta).
$$ 
We want to estimate $G_{M}^{\left(  j\right)}   (\eta)$ for every $j\in\mathbb{N}_{0}$, under the assumption that $M$ is ``large''. It follows from our previous discussion that whenever $\frac{(k + d - 1)}{2} <  \frac{ M + k}{\rho+\varepsilon},
$ and $M$ is sufficiently large, we have 
$$
G_{M}^{\left(  j\right)}   (\eta)
\le
\left(\frac{\left( e^{d^2}   C_d A_{\varepsilon}  \right)^{1/(j+1)}  C\widetilde{D} 
}{\left(  M+k+\left(  k-\beta\right)  \left(  j+1\right)
	\right)  ^{\frac{k-\beta}{\rho+\varepsilon}-\frac{k-\tau}{2}}}\right)^{j+1} 
$$
$$
\times 
\left(  M+k+\left(  k-\beta\right)  \left(  j+1\right)
	\right)^{\frac{(k + d - 1)}{2} - \frac{ M + k}{\rho+\varepsilon}}.
$$
We may assume that $ e^{d^2}   C_d A_{\varepsilon} \ge 1$ (otherwise we remove the term from the inequality). Then
$$
G_{M}^{\left(  j\right)}   (\eta)
\le
\left(\frac{\left( e^{d^2}   C_d A_{\varepsilon}  \right)  C\widetilde{D} 
}{\left(  M+k
	\right)  ^{\frac{k-\beta}{\rho+\varepsilon}-\frac{k-\tau}{2}}}\right)^{j+1} 
\left(  M+k
	\right)^{\frac{(k + d - 1)}{2} - \frac{ M + k}{\rho+\varepsilon}}.
$$
Thus we can select $M_0$ so large  that for all
$M\geq M_{0}$, 
$$
\frac{\left( e^{d^2}   C_d A_{\varepsilon}  \right)  C\widetilde{D} 
}{\left(  M+k
	\right)  ^{\frac{k-\beta}{\rho+\varepsilon}-\frac{k-\tau}{2}}}
\le
\frac{1}{2}
$$
and additionally, so that for all
$M\geq M_{0}$ the last inequality below is satisfied:
$$
|G_{M}(\eta)| 
\le
{\displaystyle\sum_{j=0}^{\infty}}
G_{M}^{\left(  j\right)}   (\eta)
\le
2 \left(  M+k
	\right)^{\frac{(k + d - 1)}{2} - \frac{ M + k}{\rho+\varepsilon}}
\le
M^{ - \frac{ M }{\rho+2 \varepsilon}}.
$$ 
\end{proof}

	\section{Fischer decompositions on certain Banach spaces of entire functions}
	
	As in \cite{KhSh92}, we use  $\Lambda $ to denote the set of all decreasing sequences 
	$
	\lambda =\left( \lambda _{m}\right) _{m\in \mathbb{N}_{0}}$ of positive
	numbers $\lambda _{m} \le 1$ which converge 	 to $0.$ Then $B_{\lambda }$ is defined
	as the space of all entire functions $f$ on $\mathbb{C}^{n}$ such that the
	homogeneous expansion $f=\sum_{m=0}^{\infty }f_{m}$ satisfies 
	\begin{equation}\label{Blambda}
		\frac{\left\Vert f_{m}\right\Vert _{a}}{m^{m/2}\left( \lambda _{m}\right)
			^{m}}\rightarrow 0.
	\end{equation}
	One can show that $B_{\lambda }$ is a Banach space with respect to the norm%
	\begin{equation}\label{Blambdanorm}
		\left\Vert f\right\Vert _{\lambda }:=\sup_{m\in \mathbb{N}_{0}}\frac{%
			\left\Vert f_{m}\right\Vert _{a}}{m^{m/2}\lambda _{m}^{m}}.
	\end{equation}
Actually, the assumption  $\lambda_m \le 1$ is not made in \cite{KhSh92}, but it will be convenient for us later on, so we include it into the definition; clearly replacing a decreasing sequence 
$
\lambda =\left( \lambda _{m}\right) _{m\in \mathbb{N}_{0}}$ with the decreasing sequence 
$
\lambda\wedge 1 :=\left( \lambda _{m}\wedge 1 \right) _{m\in \mathbb{N}_{0}}$ 
leads to the same space $B_\lambda$, with a comparable norm.

	Motivated by an anonymous referee's comments, we show that certain modifications of the preceding arguments  allow us to partially deal with a
	question formulated in \cite[Remark 5.2]{KhSh92}: can the \emph{uniqueness} assumption  in \cite[Theorem 3]{KhSh92} be omitted? The answer is yes. We will assume that the sequence $\lambda$ converges to 0 sufficiently fast, and more precisely, that condition   (\ref{eqconverging0}) below is satisfied. This condition adapts hypothesis (4.10) of \cite[Lemma 12]{KhSh92} to the more general setting considered here, where $\tau$ can be larger than 1 and $\beta$ larger than 0.
	
	\begin{theorem}
		Let $P_{k}$ be a homogeneous polynomial of degree $k>0$ on $\mathbb{C}^{d}$,
		and let us write $T:=T_{P_{k}}$. Suppose that there exist a $C>0$ and a $\tau
		\in \{0,\dots ,k\}$ such that for every $m>0$ and every homogeneous
		polynomial $f_{m}$ of degree $m,$ the following inequality holds: 
		\begin{equation*}
			\left\Vert Tf_{m}\right\Vert _{a}\leq \frac{C}{m^{\tau /2}}\left\Vert
			f_{m}\right\Vert _{a}.
		\end{equation*}%
		Assume that for $0\leq j<k$ the polynomials $P_{j}\left( z\right) $ are
		homogeneous of degree $j$, and for some $\beta <k$ and every $j$ with $\beta
		<j<k$ we have $P_{j}=0$. Let $\lambda \in \Lambda $ be such that 	
		\begin{equation}
			\lim_{m\rightarrow \infty }m^{\frac{(k-\tau )}{2}}\lambda _{m}^{(k-\beta
				)}=0.  \label{eqconverging0}
		\end{equation}

		Then for every $f\in B_{\lambda }$, there exist $q\in B_{\lambda }$ and  $r$ 
		entire  such that 
		\begin{equation*}
			f=\left( P_{k}-P_{\beta }-\cdots -P_{0}\right) q+r\text{ and }P_{k}^{\ast
			}\left( D\right) r=0.
		\end{equation*}
	\end{theorem}

	It is not clear to us whether
	$\left( P_{k}-P_{\beta }-\cdots -P_{0}\right) q$, and hence $r$, must be in $B_\lambda$. The question whether $r \in B_\lambda$ seems to be  a delicate one: it is not neccesarily true that the  product of a function in $B_\lambda$ with a polynomial must again be in $B_\lambda$.

	\begin{proof}
		We proceed as in the proof of Theorem \ref{main}: our strategy is to show that 
		\begin{equation}
			g:={\displaystyle\sum_{m=0}^{\infty }}T_{P}\left( f_{m}\right)  \label{eq16a}
		\end{equation}%
		defines an entire function $g:\mathbb{C}^{d}\rightarrow \mathbb{C}$ which belongs to $
		B_{\lambda },$ by writing $g\left( z\right) ={\displaystyle
			\sum_{M=0}^{\infty }}G_{M}\left( z\right)$, where each $G_{M}$ is a
		homogeneous polynomial of degree $M$, and then proving that $\|g\|_\lambda < \infty$ by showing that 
		\begin{equation*}
			\frac{\left\Vert G_{M}\right\Vert _{a}}{M^{M/2}\left( \lambda _{M}\right)
				^{M}}\rightarrow 0.
		\end{equation*}
	
		It has been noted in the proof of Theorem \ref{main} that 
		\begin{equation*}
			T_{P}\left( f_{m}\right) =\sum_{j=-1}^{\infty }\sum_{s_{0}=0}^{\beta
			}\sum_{s_{1}=0}^{\beta }\cdots \sum_{s_{j}=0}^{\beta }TP_{s_{j}}\cdots
			TP_{s_{0}}Tf_{m}.
		\end{equation*}%
		In order to show that the sum in (\ref{eq16a}) defines an entire function it
		suffices to prove that 
		\begin{equation*}
			G=\sum_{j=-1}^{\infty }\sum_{s_{0}=0}^{\beta }\sum_{s_{1}=0}^{\beta }\cdots
			\sum_{s_{j}=0}^{\beta }{\displaystyle\sum_{m=0}^{\infty }}TP_{s_{j}}\cdots
			TP_{s_{0}}Tf_{m}
		\end{equation*}%
		does so. Then the sum can be reordered and shown to be equal to $g.$ As
		before (cf. \ref{gm}) we collect all summands having degree $M\geq 0$ and we consider the
		sum
		
		\begin{equation}
			G_{M}(z):=\sum_{j=-1}^{\infty }\sum_{s_{0}=0}^{\beta }\sum_{s_{1}=0}^{\beta
			}\cdots \sum_{s_{j}=0}^{\beta }TP_{s_{j}}\cdots TP_{s_{0}}Tf_{M+k\left(
				j+2\right) -\left( s_{0}+\cdots +s_{j}\right) }(z).  \label{gma}
		\end{equation}%
		Next we show that $G_{M}$ converges absolutely everywhere. As in the
		proof of Theorem \ref{main} and with the same notation (cf. \ref{sjm}) setting  	$$
		S_{j,M}:=\left\Vert TP_{s_{j}}\cdots TP_{s_{0}}Tf_{M+k\left( j+2\right)
			-\left( s_{0}+\cdots +s_{j}\right) }\right\Vert _{a}$$ 
		we see that 
		\begin{equation}\label{apolyisapoly}
			\left\vert G_{M}\left( z\right) \right\vert \leq \frac{\left\vert
				z\right\vert ^{M}}{\sqrt{M!}}\sum_{j=-1}^{\infty }\sum_{s_{0}=0}^{\beta
			}\sum_{s_{1}=0}^{\beta }\cdots \sum_{s_{j}=0}^{\beta }S_{j,M}.
		\end{equation}
		Thus (\ref{gma}) converges absolutely to a homogeneous polynomial of degree $M$ if 
		\begin{equation}
		 \sum_{j=-1}^{\infty
			}\sum_{s_{0}=0}^{\beta }\sum_{s_{1}=0}^{\beta }\cdots \sum_{s_{j}=0}^{\beta
			}S_{j,M}  \label{ggmm}.
		\end{equation}
		converges.  As in the proof of Theorem \ref{main} (and with the same notation) we have
		\begin{equation*}
			S_{j,M}\leq C^{j+1}D_{s_{j}}\cdots D_{s_{0}} \ B_{-1}^{2^{-1}
				\sum_{n=0}^{j}(s_{n}-\tau )} \left\Vert f_{M+k\left(
				j+2\right) -\sum_{n=0}^{j}s_{n}}\right\Vert _{a}.
		\end{equation*}
		 Recall that 
	\begin{equation*}
		B_{-1}=M+k\left( j+2\right) -\sum_{n=0}^{j}s_{n}\geq M+k  + \left( k-\beta
		\right) \left( j+1\right).
	\end{equation*}
Let us fix $M \ge 0$ and write	$\widetilde{D}:=D_{s_{0}}+\cdots+D_{s_{\beta}}$. We want to prove the convergence of (\ref{ggmm}). Now
\begin{equation*}
	S_{j,M}
\leq
 C^{j+1} D_{s_{j}}\cdots D_{s_{0}} \ 
	B_{-1}^{2^{-1}\sum_{n=0}^{j}(s_{n}-\tau )} \ \frac{
		\left\Vert f_{B_{-1}}\right\Vert_{a}}
	{B_{-1}^{B_{-1}/2}\lambda_{B_{-1}}^{B_{-1}}} \ 
	B_{-1}^{B_{-1}/2} 
	\lambda_{B_{-1}}^{B_{-1}}.
\end{equation*}
\begin{equation*}
=
	C^{j+1} D_{s_{j}}\cdots D_{s_{0}} \ 
	\frac{\left\Vert f_{B_{-1}}\right\Vert_{a}}
	{B_{-1}^{B_{-1}/2}\lambda_{B_{-1}}^{B_{-1}}} \ 
	B_{-1}^{2^{-1} ( M+k  + \left( k-\tau
		\right) \left( j+1\right))} 
	\lambda_{B_{-1}}^{B_{-1}}.
\end{equation*}
Since  $\lambda_m \le 1$ for every $m$, and $f \in B_\lambda$, it follows that 
\begin{equation*}
\sum_{s_{0}=0}^{\beta }\sum_{s_{1}=0}^{\beta }\cdots \sum_{s_{j}=0}^{\beta} 	S_{j,M}
\le	
C^{j+1} \widetilde{D}^{j+1} \ 
\left\Vert f\right\Vert _{\lambda } \ 
B_{-1}^{2^{-1} ( M+k  + \left( k-\tau
	\right) \left( j+1\right))} 
\lambda_{B_{-1}}^{ M+k  + \left( k-\beta
	\right) \left( j+1\right)}.
\end{equation*}
Recalling the  hypothesis 
$
	\lim_{m\rightarrow \infty }m^{\frac{(k-\tau )}{2}}\lambda _{m}^{(k-\beta
		)}=0
	$, cf. (\ref{eqconverging0}), 
we see that
	\begin{equation*}
	\lim_{j\rightarrow \infty } B_{-1}^{2^{-1} \left(\frac{M+k}{j + 1}  + \left( k-\tau
		\right)\right)} 
	\lambda_{B_{-1}}^{\frac{M+k}{j + 1}  + \left( k-\beta
		\right)} =0, 
\end{equation*}
since 
$\lim_{j\rightarrow \infty } B_{-1}^{ \frac{1}{j + 1}} 
 =1$
  and 
  $\sup_{j} 
\lambda_{B_{-1}}^{\frac{1}{j + 1}} \le 1$.
Choosing $J\gg 1$ so that for every $j \ge J$ we have 
$$
C \widetilde{D} \  B_{-1}^{2^{-1} \left(\frac{M+k}{j + 1}  + \left( k-\tau
	\right)\right)} 
\lambda_{B_{-1}}^{\frac{M+k}{j + 1}  + \left( k-\beta
	\right)} < 1/2,
$$
it becomes clear that the series (\ref{ggmm}) converges,
and hence, that each $G_{M}$ is a
homogeneous polynomial of degree $M$. 
Once we know $G_{M}$ is a  polynomial, we can estimate its apolar norm using (\ref{gma}) and the triangle inequality:
\begin{equation}
	\left\Vert G_{M}\right\Vert _{a}
	\le	
	\sum_{j=-1}^{\infty
	}\sum_{s_{0}=0}^{\beta }\sum_{s_{1}=0}^{\beta }\cdots \sum_{s_{j}=0}^{\beta
	}S_{j,M}.
\end{equation}
The next step consists in  showing that 
\begin{equation*}
	\frac{\left\Vert G_{M}\right\Vert _{a}}{M^{M/2}\left( \lambda _{M}\right)
		^{M}}\rightarrow 0.
\end{equation*}
Arguing as in the proof of Theorem \ref{main}  we obtain
\begin{equation*}
	S_{j,M}\leq C^{j+1}D_{s_{j}}\cdots D_{s_{0}}\frac{B_{-1}^{2^{-1}%
			\sum_{n=0}^{j}(s_{n}-\tau )}}{B_{j}^{\tau /2}}\left\Vert f_{M+k\left(
		j+2\right) -\sum_{n=0}^{j}s_{n}}\right\Vert _{a},
\end{equation*}
 the only difference being the additional factor $1/B_{j}^{\tau /2}$,  which previously was estimated by $1$. This factor becomes important here, since $M\to \infty$ in this part of the argument, while previously $M$ was fixed.

		It follows that 
		\begin{equation*}
			\frac{S_{j,M}}{M^{M/2}\lambda _{M}^{M}}\leq C^{j+1}D_{s_{j}}\cdots D_{s_{0}}%
			\frac{B_{-1}^{2^{-1}\sum_{n=0}^{j}(s_{n}-\tau )}}{B_{j}^{\tau /2}}\frac{
				\left\Vert f_{B_{-1}}\right\Vert _{a}}{M^{M/2}\lambda _{M}^{M}}\frac{
				B_{-1}^{B_{-1}/2} \lambda _{B_{-1}}^{B_{-1}}}{B_{-1}^{B_{-1}/2}\lambda _{B_{-1}}^{B_{-1}}}.
		\end{equation*}%
As before,
		\begin{equation*}
			B_{-1}^{2^{-1}\sum_{n=0}^{j}(s_{n}-\tau
				)}B_{-1}^{B_{-1}/2}=B_{-1}^{M/2+k/2+\left( k-\tau \right) \left( j+1\right)
				/2},
		\end{equation*}%
		so using  $B_{j} = M+k$ and \begin{equation*}
			\frac{\left\Vert f_{B_{-1}}\right\Vert _{a}}{B_{-1}^{B_{-1}/2}\lambda
				_{B_{-1}}^{B_{-1}}}
			\leq
			\left\Vert f\right\Vert _{\lambda },
		\end{equation*}
	 we get
		\begin{equation*}
			\frac{S_{j,M}}{M^{M/2}\lambda _{M}^{M}}\leq C^{j+1}D_{s_{j}}\cdots D_{s_{0}}%
			\frac{B_{-1}^{M/2+k/2+\left( k-\tau \right) \left( j+1\right) /2}}{\left(
				M+k\right) ^{\tau /2}M^{M/2}}\frac{\lambda _{B_{-1}}^{B_{-1}}}{\lambda
				_{M}^{M}}
				\left\Vert f\right\Vert _{\lambda }.
		\end{equation*}%
	  For every $M \ge 0$, since $B_{-1}\geq M$ and all terms in $\lambda$ are bounded by 1, we have 
$\lambda _{B_{-1}}\leq \lambda _{M}$ and 
		\begin{equation*}
			\frac{\lambda _{B_{-1}}^{B_{-1}}}{\lambda _{M}^{M}}=\left( \frac{\lambda
				_{B_{-1}}}{\lambda _{M}}\right) ^{M}\left( \lambda _{B_{-1}}\right)
			^{k\left( j+2\right) -\sum_{n=0}^{j}s_{n}}\leq \left( \lambda
			_{B_{-1}}\right) ^{k+\left( k-\beta \right) \left( j+1\right) }.
		\end{equation*}
		Furthermore,  if $M\ge 1$ then
		\begin{equation*}
			\left( \frac{B_{-1}}{M}\right) ^{M/2}\leq \left( 1+\frac{k\left( j+2\right) 
			}{M}\right) ^{M/2}\leq e^{\frac{1}{2}k\left( j+2\right) },
		\end{equation*}
	so
		\begin{equation*}
			\frac{S_{j,M}}{M^{M/2}\lambda _{M}^{M}}\leq C^{j+1}D_{s_{j}}\cdots D_{s_{0}}
			\frac{B_{-1}^{k/2+\left( k-\tau \right) \left( j+1\right) /2}}{\left(
				M+k\right) ^{\tau /2}} \ e^{\frac{1}{2}k\left( j+2\right) }\left( \lambda
			_{B_{-1}}\right) ^{k+\left( k-\beta \right) \left( j+1\right) }
			 \left\Vert f\right\Vert _{\lambda }.
		\end{equation*}%
		Next we observe that 
		\begin{eqnarray*}
			\frac{B_{-1}^{\tau /2}}{\left( M+k\right) ^{\tau /2}} &\leq &\left( \frac{%
				M+k\left( j+2\right) }{M+k}\right) ^{\tau /2}=\left( 1+\frac{k\left(
				j+1\right) }{M+k}\right) ^{\tau /2} \\
			&\leq &\left( 1+\left( j+1\right) \right) ^{\tau /2}\leq \left(
			1+2^{j}\right) ^{\tau /2}\leq 2^{\left( j+1\right) \tau /2}.
		\end{eqnarray*}
	 Thus 
		\begin{eqnarray*}
			\frac{S_{j,M}}{M^{M/2}\lambda _{M}^{M}} &
			\leq &C^{j+1}D_{s_{j}}\cdots
			D_{s_{0}}e^{\frac{1}{2}  k \left( j+2\right) } \ 2^{\left( j+1\right) \tau /2} \ B_{-1}^{\left( k-\tau \right)
				/2+\left( k-\tau \right) \left( j+1\right) /2}\left( \lambda
			_{B_{-1}}\right) ^{k+\left( k-\beta \right) \left( j+1\right) } \left\Vert f\right\Vert _{\lambda } \\
			&= &D_{s_{j}}\cdots D_{s_{0}}e^{\frac{1}{2}k}\left( Ce^{\frac{1}{2}%
				k}2^{\tau /2}B_{-1}^{\left( k-\tau \right) /2}\left( \lambda
			_{B_{-1}}\right)^{\left( k-\beta \right) }\right)^{j+1}B_{-1}^{\left(
				k-\tau \right) /2}\left( \lambda _{B_{-1}}\right) ^{k} \left\Vert f\right\Vert _{\lambda }.
		\end{eqnarray*}%
Writing again	$\widetilde{D}:=D_{s_{0}}+\cdots+D_{s_{\beta}}$,   by (\ref{eqconverging0}),  for every 
$M \gg 1$ sufficiently large we have 
		\begin{equation*}
			Ce^{\frac{1}{2}k}2^{\tau /2}\cdot B_{-1}^{\frac{(k-\tau )}{2}}\lambda
			_{B_{-1}}^{(k-\beta )}\leq \frac{1}{2\widetilde{D}},
		\end{equation*}
		from whence it follows that 
		\begin{equation*}
			\frac{S_{j,M}}{M^{M/2}\lambda _{M}^{M}}\leq D_{s_{j}}\cdots D_{s_{0}}e^{%
				\frac{1}{2}k}\left( \frac{1}{2\widetilde{D}}\right) ^{j+1}B_{-1}^{\left(
				k-\tau \right) /2}\left( \lambda _{B_{-1}}\right) ^{k}\left\Vert
			f\right\Vert _{\lambda }.
		\end{equation*}

Therefore
		\begin{equation*}
			\frac{\left\Vert G_{M}\right\Vert _{a}}{M^{M/2}\lambda _{M}^{M}}\leq
			\sum_{j=-1}^{\infty }\sum_{s_{0}=0}^{\beta }\sum_{s_{1}=0}^{\beta
			}\cdots \sum_{s_{j}=0}^{\beta }\frac{S_{j,M}}{M^{M/2}\lambda _{M}^{M}}\leq
			\left\Vert f\right\Vert _{\lambda }B_{-1}^{\left(
				k-\tau \right) /2}\left( \lambda _{B_{-1}}\right) ^{k}e^{\frac{1}{2}k}\sum_{j=-1}^{\infty }%
			\frac{1}{2^{j+1}}.
		\end{equation*}%
	Recalling (\ref{eqconverging0}), 
	we see that
	\begin{equation*}
		\lim_{M\rightarrow \infty } B_{-1}^{2^{-1}  \left( k-\tau
			\right)} 
		\lambda_{B_{-1}}^{k} =0, 
	\end{equation*}
	and thus
		\begin{equation*}
			\frac{\left\Vert G_{M}\right\Vert _{a}}{M^{M/2}\lambda _{M}^{M}}\rightarrow 0.
		\end{equation*}
	\end{proof}

\smallskip \noindent \textbf{Data availability} Data sharing is not
applicable to this article since no data sets were generated or analyzed.

\smallskip \noindent \textbf{Declarations}

\smallskip \noindent \textbf{Conflict of interest} The authors declare that
they have no competing interests.

\smallskip \noindent \textbf{Acknowlegement}  \thanks{The first named author was partially supported by Grant PID2019-106870GB-I00 of the MICINN of Spain,  by ICMAT Severo Ochoa project
CEX2019-000904-S (MICINN), and by
the Madrid Government (Comunidad de Madrid - Spain)  V PRICIT (Regional Programme of Research and Technological Innovation), 2022-2024.}


\begin{thebibliography}{99}


\bibitem{Armi04}  Armitage, D.:  The Dirichlet problem when the boundary
	function is entire.  J. Math. Anal. \textbf{291},  565--577 (2004) 
	


	\bibitem{Barg61}  Bargmann, V.:  On a Hilbert space of analytic functions and an associated integral transform.  Comm. Pure Appl. Math. \textbf{14},   187--214 (1961) 
	
	\bibitem{Beau97} Beauzamy, B.: Extremal products in Bombieri's norm. 
	Rend. Istit. Mat. Univ. Trieste, Suppl. Vol. XXVIII  73--89  (1997), 
	

	\bibitem{EbSh95} Ebenfelt, P.  Shapiro, H.S.:   The Cauchy--Kowaleskaya
		theorem and Generalizations.  Commun. Partial Differential Equations \textbf{20} 
	 939--960 (1995) 
	
	\bibitem{EbSh96} Ebenfelt, P.  Shapiro, H.S.:    The mixed Cauchy problem
		for holomorphic partial differential equations. J. D'Analyse Math. \text{65}, 
	237--295 (1996)  
	
	 
	
	\bibitem{EbRe08} Ebenfelt, P., Render, H.:  The mixed Cauchy problem with
		data on singular conics.   J. London Math. Soc. \textbf{78}, 248--266   (2008)
	
	\bibitem{EbRe08b} Ebenfelt, P., Render, H.:   The Goursat Problem for a Generalized Helmholtz Operator in $R^{2}. $ Journal D'Analyse Math. \textbf{105}, 149--168  (2008) 

	
	\bibitem{ElNa12}  Elizar'ev, I.N., Napalkov, V.V.: Fischer Decomposition
		for a Class of Inhomogeneous Polynomials.  Doklady Mathematics, 2012, Vo.
	85, No. 2, pp. 196--197, see also Doklady Akademii Nauk, 2021, Vol 443, no.
	2, pp. 156--157.
	

	

	  
	\bibitem{KhSh92} Khavinson, D., Shapiro,  H.S.: Dirichlet's problem when
		the data is an entire function.  Bull. London Math. Soc. \textbf{24}, 456--468 (1992)
		
		
	\bibitem{LuRe11} Lundberg, E.,  Render, H.:  The  Khavinson-Shapiro
		conjecture and polynomial decompositions.   J. Math. Anal. Appl.  \textbf{ 376}, 
	506--513 (2011) 
	
	
	
	 
	\bibitem{MeSt85}  Meril, A., Struppa, D.:  Equivalence of Cauchy problems
		for entire and exponential type functions.  Bull. Math. Soc. \textbf{17},
	469--473 (1985) 
	
	\bibitem{MeYg92}    Meril, A., Yger, A.A.:  Probl\`{e}mes de Cauchy globaux. 
	(French) [Global Cauchy problems] Bull. Soc. Math. France \textbf{120}, no. 1, 87--111 (1992). 
	

	
	\bibitem{NeSh66} Newman, D.J.,  Shapiro,  H.S.:  Certain Hilbert spaces of
		entire functions.  Bull. Amer. Math. Soc.  \textbf{ 72}, 971--977 (1966) 
	
	

	\bibitem{Rend08}  Render, H.: Real Bargmann spaces, Fischer
		decompositions and sets of uniqueness for polyharmonic functions.  Duke
	Math. J. \textbf{142}, 313--352  (2008)
	
	 

	
	\bibitem{Rend16}  Render, H.: A characterization of the
		Khavinson-Shapiro conjecture via Fischer operators.  Potential Anal. \textbf{45}, 
	 539--543 (2016) 
	
		\bibitem{RendAl22} Render, H., Aldaz,  J. M.:  Fischer decompositions for entire functions and the Dirichlet
		problem for parabolas.   Anal. Math. Phys. \textbf{12} , no. 6, Paper No. 150, 29 pp. (2022) 
	

	\bibitem{Shap89} Shapiro, H.S.: An algebraic theorem of E. Fischer and
		the Holomorphic Goursat Problem.  Bull. London Math. Soc. \textbf{21},
	513--537 (1989)
	

	
	\bibitem{Zeil} Zeilberger, D.:   Chu's identity implies Bombieri's 1990
		norm-inequality.  Amer. Math. Monthly \textbf{101}, 894--896  (1994)
\end{thebibliography}
\end{document}